\documentclass[12pt]{amsart}

\title[Distance and the Goeritz groups of bridge decompositions]
{Distance and the Goeritz groups of bridge decompositions}

\author{Daiki Iguchi}
\address{
Department of Mathematics \newline
\indent Hiroshima University, 1-3-1 Kagamiyama, Higashi-Hiroshima, 739-8526, Japan}
\email{d200643@hiroshima-u.ac.jp}

\author{Yuya Koda}
\address{
Department of Mathematics \newline
\indent Hiroshima University, 1-3-1 Kagamiyama, Higashi-Hiroshima, 739-8526, Japan}
\email{ykoda@hiroshima-u.ac.jp}

\usepackage{amsthm}
\usepackage{mathrsfs}
\usepackage{latexsym}
\usepackage[dvips]{graphicx}
\usepackage[dvips]{psfrag}
\usepackage[dvips]{color}
\usepackage{xypic}
\usepackage[abs]{overpic}
\usepackage{caption}
\usepackage[all]{xy}
\usepackage[normalem]{ulem}

\theoremstyle{plain}
\newtheorem*{theorem*}{Theorem}
\newtheorem*{lemma*} {Lemma}
\newtheorem*{corollary*} {Corollary}
\newtheorem*{proposition*}{Proposition}
\newtheorem*{conjecture*}{Conjecture}
\newtheorem*{claim*}{Claim}
\newtheorem{theorem}{Theorem}[section]
\newtheorem{lemma}[theorem]{Lemma}

\newtheorem{proposition}[theorem]{Proposition}

\theoremstyle{remark}

\newtheorem*{definition}{Definition}

\newtheorem*{remark}{Remark}

\newtheorem*{ack}{Acknowledgements}

\theoremstyle{definition}

\newtheoremstyle{citing}
  {}
  {}
  {\itshape}
  {}
  {\bfseries}
  {.}
  {.5em}
  {\thmnote{#3}}

\theoremstyle{citing}
\newtheorem*{citingtheorem}{} 

\newcommand{\Cl}{\operatorname{Cl}}



\begin{document}

\maketitle

\begin{abstract}
We prove that if the distance of a bridge decomposition of a link with respect to a Heegaard splitting 
of a $3$-manifold is at least $6$, 
then the Goeritz group is a finite group. 
\end{abstract}





\section*{Introduction}\label{sec:introduction}

It is well known that any closed orientable $3$-manifold $M$ is obtained by gluing two handlebodies $V^+$ and $V^-$ of the same genus
along their boundaries. 
Such a decomposition of $M$ is called a {\it Heegaard splitting}, 
and the common boundary $\Sigma$ of $V^+$ and $V^-$ is called the {\it Heegaarrd surface} of the splitting. 
The {{\it distance} of a Heegaard splitting $M = V^+ \cup_\Sigma V^-$ 
is a measure of complexity introduced by Hempel \cite{Hem01}. 
It is defined to be the distance  
between the sets of meridian disks of $V^+$ and $V^-$ 
in the curve graph of $\Sigma$. 
The distance has successfully provided a way of describing the topology 
and geometry of $3$-manifolds.  

A {\it bridge decomposition} of a link $L$ in a closed orientable $3$-manifold $M$ 
is a Heegaard splitting $M = V^+ \cup_\Sigma V^-$ such that 
$L$ intersects each of $V^+$ and $V^-$ in properly embedded trivial arcs.  
When the genus of the surface $\Sigma$, called a {\it bridge surface}, is $g$, and the number of components of 
$V^\pm \cap L$ is $n$, we particularly call such a decomposition a {\it $(g, n)$-decomposition} of 
$L$.  
The distance is also defined for a bridge decomposition 
in the same way as in the case of a Heegaard splitting, and 
many results about Heegaard splittings have been extended to bridge decompositions. 
For example,  Bachman-Schleimer \cite{BS05} showed that 
the distance of a bridge decomposition of a knot 
is bounded above by the Euler characteristic 
of an essential surface in the complement of the knot, 
which is a generalization of a result of Hartshorn \cite{Hart02}. 
The arguments in \cite{BS05} apply to the case of links as well, and their results, 
in particular, imply that if the distance of a link in a $3$-manifold is at least $5$,
then the complement of the link admits a complete hyperbolic structure of finite volume. 
(The definition of the distance in this paper is slightly different from that in \cite{BS05}, see Section \ref{sec:preliminaries_dist}.) 
As another example, 
Tomova \cite{Tom07} gave a sufficient condition for uniqueness of bridge decompositions of knots 
in terms of the distances and the Euler characteristics of bridge surfaces,  
which is a generalization of a result of Sharlemann-Tomova \cite{ST06}.    

In this paper, we are interested in the Goeritz group of a bridge decomposition. 
The {\it Goeritz group} (or {\it the mapping class group}) {\it of a Heegaard splitting} 
is defined to be the group of 
isotopy classes of self-diffeomorphisms of the ambient $3$-manifold that preserve the splitting setwize. 
Namazi \cite{Nam07} showed that for each topological type of Heegaard surface 
there exists a constant $C$ such that if the distance is at least $C$,  
then the Goeritz group is a finite group. 
Johnson \cite{Joh10_PAMS} had refined this result by showing the constant $C$ can be taken to be at most $4$ 
independently of the genus of the Heegaard surface. 
The concept of Goeritz group has also been extended for bridge decompositions 
in \cite{HIKK21}. In that paper, a variation of Namazi and Johnson's result for the case 
of bridge decomposition was obtained: 
 it was shown that the constant $C$ for the finiteness of the Goeritz group can be taken uniformly to be at most $3796$.  
The main result of the paper is the following, which improves the above mentioned result of \cite{HIKK21}.  

\begin{theorem}\label{thm:main}
Let $g\geq 0$, $n >0$ and $(g,n)\not=(0,1),(0,2),(1,1)$. 
Let $(M,L;\Sigma)$ be a $(g,n)$-decomposition of a link $L$ in a $3$-manifold $M$. 
If the distance of $(M,L;\Sigma)$ is at least $6$, 
then the Goeritz group $\mathcal{G}(M,L;\Sigma)$ is a finite group.  
Further, for  
a $(0,n)$-decomposition $(S^{3},L;\Sigma)$ of a link $L$ in the $3$-sphere $S^3$, 
where $n \geq 3$,  
if the distance of $(S^{3},L;\Sigma)$ is at least $5$, 
then the Goeritz group $\mathcal{G}(S^{3},L;\Sigma)$ is a finite group.  
\end{theorem}   

%

Theorems~\ref{thm:main} is proved 
by extending the argument of \cite{Joh10_PAMS} to the case of bridge decompositions. 
In fact, the major part of the proof is devoted to show the following. 

\begin{citingtheorem}[Theorem \ref{thm:injection}]
Let $L$ be a link in a $3$-manifold $M$ and 
$(M,L)=(V^{+},V^{+} \cap L) \cup_{\Sigma} (V^{-},V^{-} \cap L)$ a bridge decomposition. 
If the distance between the sets of disks and once-punctured disks in $V^{+} - L$ and $V^{-} - L$ 
in the curve graph of $\Sigma - L$ are at least $4$, then 
the natural homomorphism $\eta:\mathcal{G}(M,L;\Sigma) \rightarrow \mathrm{MCG}(M,L)$ is injective.   
\end{citingtheorem}

\noindent 
The key tool for the proof is the {\it double sweep-out} technique involving the theory 
of graphics introduced by Rubinstein-Scharlemann \cite{RS96}. 
As noted above, 
if the distance of the bridge decomposition $(M,L;\Sigma)$ is at least $6$, then 
the complement $M - L$ admits a hyperbolic structure, and hence 
the mapping class $\mathrm{MCG}(M,L)$ is a finite group. 
Theorem~\ref{thm:main} thus follows from Theorem~\ref{thm:injection} and these facts. 

The paper is organized as follows. 
In Section \ref{sec:preliminaries}, 
we review basic definitions and properties of the distance and the Goeritz group of  
a bridge decomposition. 
In Section \ref{sec:sweep-outs}, we review the theory of sweep-outs, which is the main tool of the paper. 
In Section \ref{sec:upper_bound_for_distance}, we give the proof of Theorem~\ref{thm:injection}. 
Finally, in Section \ref{sec:proof_of_main_thm}, we give the proof of Theorem~\ref{thm:main}.

\section{Preliminaries}\label{sec:preliminaries}
We work in the smooth category. 
Throughout the paper, we will  use the following notations and conventions:  
\begin{itemize}
\item Given two sets $A$ and $B$, we denote by $A - B$ or $A_{B}$ the relative complement of $B$ in $A$.    
\item For a topological space $X$, we denote by $|X|$ the number of connected components of $X$. 
For a subset $Y \subset X$, 
we denote by $\mathrm{Cl}(Y; X)$ or $\mathrm{Cl}(Y)$ the closure of $Y$ in $X$.  
\item For simplicity, 
we will not distinguish notationally between simple closed curves in a surface and their isotopy classes  
throughout.   
\end{itemize}

\subsection{Bridge decompositions}

Let $g \geq 0$ and $n >0$. 
Let $V$ be a handleboby of genus $g$. 
The union of $n$ properly embedded, mutually disjoint arcs in $V$ is called an {\it $n$-tangle}. 
An $n$-tangle in $V$ is said to be {\it trivial} 
if the arcs can be isotoped into $\partial V$ simultaneously.  
Let $L$ be a link in a closed orientable $3$-manifold $M$. 
Let $M=V^{+} \cup_{\Sigma} V^{-}$ be a genus-$g$ Heegaard splitting of $M$. 
A decomposition $(M,L)=(V^{+}, V^{+} \cap L) \cup_\Sigma (V^{-},V^{-} \cap L)$ 
is called a {\it $(g,n)$-decomposition} of $L$  
if $V^{+} \cap L$ and $V^{-} \cap L$ are trivial $n$-tangles in $V^{+}$ and $V^{-}$, respectively.  
We sometimes denote such a decomposition by $(M,L;\Sigma)$. 
The surface $\Sigma$ here is called the {\it bridge surface of $L$}. 
Two bridge decompositions of $L$ are said to be {\it equivalent} 
if their bridge surfaces are isotopic through bridge surfaces of $L$. 

\subsection{Curve graphs}
\label{sec:preliminaries_curves}

Let $g \geq 0$ and $k >0$. 
Let $\Sigma$ be a closed orientable surface of genus $g$ with $k$ marked points 
$p_{1},p_{2},\ldots,p_{k}$. 
Set $\Sigma^{\prime}:=\Sigma - \{p_{1},p_{2},\ldots,p_{k}\}$. 
A simple closed curve in $\Sigma^{\prime}$ is said to be {\it essential} 
if it does not bound a disk or a once-punctured disk in $\Sigma^{\prime}$. 
We say that an open arc $\alpha$ in $\Sigma'$ is {\it essential} if it satisfies the following:
\begin{itemize}
\item
$\Cl(\alpha; \Sigma) - \alpha \subset \{  p_{1},p_{2},\ldots,p_{k} \}$; and 
\item
If $\Cl(\alpha; \Sigma)$ is a simple closed curve bounding a disk $D$ in $\Sigma$, 
then the interior of $D$ contains at least one point of $\{  p_{1},p_{2},\ldots,p_{k} \}$.  
\end{itemize}
The {\it curve graph $\mathcal{C}(\Sigma^{\prime})$ of $\Sigma^{\prime}$} is the graph 
whose vertices are isotopy classes of essential simple closed curves in  $\Sigma^{\prime}$, 
and the edges are pairs of vertices $\{\alpha,\beta\}$ 
with $\alpha \cap \beta=\emptyset$. 
Similarly,  the {\it arc and curve graph $\mathcal{AC}(\Sigma^{\prime})$ of $\Sigma^{\prime}$} is the graph 
whose vertices are isotopy classes of essential simple closed curves and  
essential open arcs in $\Sigma'$, 
and the edges are pairs of vertices $\{\alpha,\beta\}$ 
with $\alpha \cap \beta=\emptyset$. 
By abuse of notation, 
we denote the underlying space of the curve graph (the arc and curve graph, respectively) 
by the same symbol $\mathcal{C}(\Sigma^{\prime})$ ($\mathcal{AC}(\Sigma^{\prime})$, respectively). 
The graph $\mathcal{C}(\Sigma^{\prime})$ ($\mathcal{AC}(\Sigma^{\prime})$, respectively) 
can be viewed as a geodesic metric space 
with the simplicial metric $d_{\mathcal{C}(\Sigma^{\prime})}$ ($d_{\mathcal{AC}(\Sigma^{\prime})}$, respectively). 
We note that the curve graph $\mathcal{C}(\Sigma^{\prime})$ is non-empty and connected if and only if $3g-4+k > 0$. 

\subsection{Distances}
\label{sec:preliminaries_dist}

In this subsection, we give the definition of the distance of a bridge decomposition 
and its variations, and summarize their basic properties. 

Let $(M,L)=(V^{+}, V^{+} \cap L) \cup_{\Sigma} (V^{-},V^{-} \cap L)$ be a $(g,n)$-decomposition 
of a link $L$ in a closed orientable $3$-manifold $M$, 
where $3g - 4 + 2n > 0$. 
We denote by $\mathcal{D}(V_L^\pm)$ the set of vertices of $\mathcal{C}(\Sigma_L)$ 
that are represented by simple closed curves bounding disks in $V_L^\pm$.  

\begin{definition}
The {\it distance} $d(M,L;\Sigma)$ of the bridge decomposition $(M,L; \Sigma)$ is defined by 
$d(M,L;\Sigma) := d_{\mathcal{C}(\Sigma_{L})}(\mathcal{D}(V^{+}_{L}),\mathcal{D}(V^{-}_{L}))$.  
\end{definition}

There are other variations, $d_{\mathcal{PD}}(M,L;\Sigma)$ and 
$d_\mathit{BS}(M,L;\Sigma)$, of the distance. 
The first one, $d_{\mathcal{PD}}(M,L;\Sigma)$, is defined as follows. 
Let $\mathcal{PD}(V_L^\pm)$ denote the set of all vertices of 
$\mathcal{C}(\Sigma_L)$  
that are represented by simple closed curves bounding disks in $V^\pm$ 
that intersect $L$ at most once. 
Then $d_{\mathcal{PD}}(M,L;\Sigma)$ is defined by 
$d_{\mathcal{PD}}(M,L;\Sigma):= d_{\mathcal{C}(\Sigma_{L})}(\mathcal{PD}(V^{+}_{L}),\mathcal{PD}(V^{-}_{L}))$. 
It is easily checked that the following inequality holds:  
\begin{eqnarray}\label{eq:d_PD}
d_{\mathcal{PD}}(M,L;\Sigma) \le d(M,L;\Sigma) \le d_{\mathcal{PD}}(M,L;\Sigma)+2. 
\end{eqnarray}

Furthermore, for a $(0,n)$-decomposition the following holds.  

\begin{proposition}[Jang {\cite[Proposition $1.2$]{Jan14}}]
\label{prop:d_PD}
Suppose that $M=S^3$ and the genus of $\Sigma$ is zero. Then, 
\begin{itemize} 
\item $d_{\mathcal{PD}}(M,L;\Sigma) = d(M,L;\Sigma)$ if $d_{\mathcal{PD}}(M,L;\Sigma) \ge 1$, and 
\item $d(M,L;\Sigma)=0$ or $1$ if $d_{\mathcal{PD}}(M,L;\Sigma) =0$. 
\end{itemize}
\end{proposition}

We next define $d_{\mathit{BS}}(M,L;\Sigma)$, which was introduced by Bachman-Schleimer \cite{BS05}. 
For trivial $n$-tangles $(V^\pm,V^\pm \cap L)$, 
we denote by $\mathcal{B}(V^\pm,V^\pm \cap L)$ the set of all vertices $\alpha$ of $\mathcal{AC}(\Sigma_L)$ 
such that  
\begin{itemize}
\item $\alpha \in \mathcal{PD}(V_L^\pm)$, or 
\item $\alpha$ is an open arc in $\Sigma_L$ such that 
$\partial \Cl (\alpha; \Sigma) \subset \Sigma \cap L$ and 
	$\Cl (\alpha; \Sigma)$ cobounds a disk in $V^\pm$ with 
	an arc of $V^\pm \cap L$. 
\end{itemize}
We define $d_{\mathit{BS}}(M,L;\Sigma)$ by the distance between two set 
$\mathcal{B}(V^+,V^+ \cap L)$ and $\mathcal{B}(V^-,V^- \cap L)$ 
in the arc and curve graph $\mathcal{AC}(\Sigma_L)$. 
By the argument of the proof of the inequality $(1)$ in p.480 of 
Korkmaz-Papadopoulos \cite{KP10}, we have 
\begin{eqnarray}\label{eq:d_BS}
\frac{1}{2}d(M,L;\Sigma) \le d_\mathrm{BS}(M,L;\Sigma) \le d(M,L;\Sigma).  
\end{eqnarray}
See also \cite{BCJTT17}. 
We summarize a few facts needed in Sections \ref{sec:upper_bound_for_distance} and \ref{sec:proof_of_main_thm}.  
The following lemma is an extension of Haken's lemma \cite{Hak68}. 

\begin{lemma}[{\cite[Lemma $4.1$]{BS05}}]\label{lem:Haken}
Let $(M,L;\Sigma)$ be a bridge decomposition of a link $L$ in a closed orientable $3$-manifold $M$. 
If $M_L$ contains an essential $2$-sphere, or 
if there exists a $2$-sphere in $M$ that intersects $L$ transversely at a single point, 
then $d_{\mathcal{PD}}(M,L;\Sigma)=0$.  
\end{lemma}
\begin{remark}
\cite[Lemma 4.1]{BS05} is stated only for knots, but their arguments hold for links.
\end{remark}

Corollary $6.2$ of Bachman-Schleimer \cite{BS05} says that 
if $d_{\mathrm{BS}}(M,L;\Sigma) \geq 3$, the complement of $L$ admits a complete hyperbolic 
structure of finite volume. 
(Again, \cite[Corollary 6.2]{BS05} is stated for knots, but their arguments are valid even for links.) 
Combining this fact and the inequality (\ref{eq:d_BS}), we have the following.  

\begin{theorem}\label{thm:hyperbolic_structure}
Let $(M,L;\Sigma)$ be a bridge decomposition of a link $L$ in a closed orientable $3$-manifold $M$. 
If $d(M,L;\Sigma) \geq 5$, then $M_L$ admits a complete hyperbolic 
structure of finite volume. 
\end{theorem}

%

\begin{remark}
Here is a subtle remark on the various notion of distances introduced above. 
In \cite{Jan14}, two notions of distance of bridge decompositions are discussed. 
One is $d$, which is denoted by $d_T$ in \cite{Jan14}, and the other is $d_{\mathcal{PD}}$, 
which is denoted by $d_{\mathit{BS}}$ in the same paper. 
The important thing to note is that the definition of $d_{\mathit{BS}}$ in \cite{Jan14} 
is different from the original one by Bachman-Schleimer \cite{BS05}. 
Then, in \cite[Theorem 5.1]{IM17}, it is claimed that if $(S^3,L;\Sigma)$ is 
a $(0,n)$-decomposition of a link $L$ in $S^3$, where 
$n \geq 3$, then the complement of $L$ admits a complete hyperbolic structure of finite volume. 
The proof in that paper bases on two results. 
One is \cite[Corollary 6.2]{BS05}. 
The other is, however, not a relationship between $d$ and $d_{\mathit{BS}}$ 
but Proposition \ref{prop:d_PD} above (literally this is described as 
a relationship between $d_T$ and $d_{\mathit{BS}}$ in \cite{Jan14}). 
Thus, we do not have a reasonable explanation of \cite[Theorem 5.1]{IM17}. 
If \cite[Theorem 5.1]{IM17} is still valid, then 
we can improve the distance estimation of Theorem \ref{thm:main} 
for $(0,n)$-decompositions of links in $S^3$. 
\end{remark}


\subsection{Goeritz groups}

Let $M$ be an orientable manifold, and $Y_{1},Y_{2}\ldots,Y_{k}$ (possibly empty) subsets of $M$. 
Let $\mathrm{Diff}(M,Y_{1},Y_{2},\ldots,Y_{k})$ denote the group of 
orientation-preserving self-diffeomorphisms of $M$ 
that send $Y_{i}$ to itself for $i=1,2,\ldots,k$. 
The {\it mapping class group} $\mathrm{MCG}(M,Y_{1},\ldots,Y_{k})$ of the 
$(k+1)$-tuple $(M,Y_{1},Y_{2}\ldots,Y_{k})$ 
is defined to be the group of connected components of $\mathrm{Diff}(M,Y_{1},Y_{2}\ldots,Y_{k})$. 

\begin{definition}
For a bridge decomposition $(M,L)=(V^{+}, V^{+} \cap L) \cup_{\Sigma} (V^{-},V^{-} \cap L)$, 
the mapping class group $\mathrm{MCG}(M,V^{+},L)$ is called the {\it Goeritz group}, and 
it is denoted by $\mathcal{G}(M,L;\Sigma)$.   
\end{definition}

Let $(M,L;\Sigma)$ be a bridge decomposition of a link $L$ in a closed orientable $3$-manifold $M$. 
Since the natural map 
$\mathcal{G}(M,L;\Sigma) \rightarrow \mathrm{MCG}(\Sigma,\Sigma \cap L)$ 
obtained by restricting the maps of concern to $\Sigma$ is injective, 
the Goeritz group can be thought of a subgroup of $\mathrm{MCG}(\Sigma,\Sigma \cap L)$.  
Thus, we can write as 
$$\mathcal{G}(M,L;\Sigma)=\mathrm{MCG}(V^{+},V^{+} \cap L) \cap \mathrm{MCG}(V^{-},V^{-} \cap L) \subset 
\mathrm{MCG}(\Sigma,\Sigma \cap L).$$

\section{Sweep-outs}\label{sec:sweep-outs}
We review the basic theory of sweep-outs.   
The main references of this section are Kobayashi-Saeki \cite{KS00}
 and Johnson \cite{Joh10_JTP}.  
In the following,  
let $M$ be a closed orientable $3$-manifold, 
$L$ a link in $M$, 
and $(M,L;\Sigma)$ a $(\mathrm{genus}(\Sigma),n)$-decomposition of $L$ throughout.

\begin{definition}
A function $f:M \rightarrow [-1,1]$ is said to be a {\it sweep-out 
of $(M, L)$ associated with the decomposition $(M, L; \Sigma)$} if 
\begin{itemize}
\item for all $s \in (-1,1)$, $f^{-1}(s)$ is a bridge surface of $L$ 
	and the bridge decomposition $(M,L;f^{-1}(s))$ is equivalent to 
	$(M,L;\Sigma)$; and 
\item $f^{-1}(1)$ and $f^{-1}(-1)$ are finite graphs, which are called {\it spines}, embedded in $M$. 
\end{itemize}
\end{definition}

We note that any bridge decomposition admits a sweep-out.  
For simplicity, we shall always assume further that the spines $f^{-1}(\pm 1)$ are uni-trivalent graphs, 
and the intersection of the spines and $L$ is exactly the set of vertices whose valency is one. 
See Figure~\ref{fig:sweeping_out_by_bridge_surface}. 

\begin{figure}[htbp]
\begin{center}
\includegraphics[width=5cm,clip]{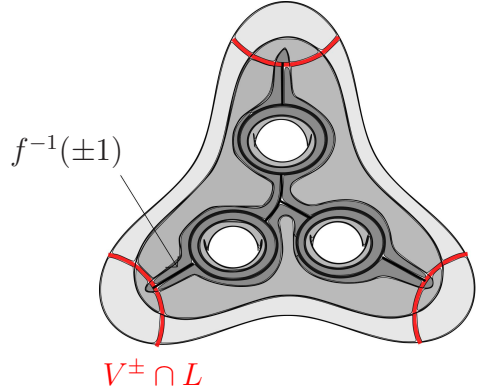}
\begin{picture}(400,0)(0,0)
\put(95,85){$f^{-1}(\pm 1)$}
\put(130,0){{\color{red}$V^{\pm} \cap L$}}
\end{picture}
\caption{A spine in $V^{\pm}$.}
\label{fig:sweeping_out_by_bridge_surface}
\end{center}
\end{figure}

A smooth map $F$ from a $3$-manifold $N$ into $\mathbb{R}^{2}$ is said to be {\it stable} 
if there exists a neighborhood $U(F)$ of $F$ 
in the space of smooth maps $C^{\infty}(N,\mathbb{R}^{2})$ with the following property: 
for any $G \in U(F)$, 
there exist diffeomorphisms 
$\varphi:N \rightarrow N$ and $\psi:\mathbb{R}^{2} \rightarrow \mathbb{R}^{2}$ 
satisfying $G \circ \varphi=\psi \circ F$. 
The image of the set of singular points of the stable map is called the {\it discriminant set}. 

Let $f$ and $g$ be sweep-outs of $(M, L)$. 
Due to Kobayashi-Saeki \cite{KS00}, 
the map $f \times g:M \rightarrow [-1,1] \times [-1,1]$ can be perturbed 
so that $f \times g$ is stable in the complement of spines of $f$ and $g$. 
In the following, whenever we consider the product of sweep-outs, 
we slightly perturb it to be stable. 
Let $\Gamma$ be the closure in $ [-1,1] \times [-1,1]$ of 
the union of the discriminant set of $f \times g$ and the image of $L$ under the map $f \times g$. 
Then $\Gamma$ is naturally equipped with a structure of a finite graph of valency at most four.
Such a finite graph is called the {\it (Rubinstein-Scharlemann) graphic defined by $f \times g$}.  

Each point $(s,t)$ in the interior of the square $[-1,1] \times [-1,1]$ corresponds to the intersection 
of two level surfaces $\Sigma_s := f^{-1}(s)$ and $\Sigma'_t := g^{-1}(t)$. 
We note that the surfaces $\Sigma_s$ and $\Sigma'_t$ always intersect $L$ transversely by definition. 
If the point $(s,t)$ lies in the complementary region of the graphic $\Gamma$, 
then the surfaces $\Sigma_s$ and $\Sigma'_t$ intersect transversely 
and $\Sigma_s \cap \Sigma_t \cap L = \emptyset$. 
If $(s,t)$ lies in the interior of an edge of $\Gamma$, 
then either 
\begin{itemize}
\item 
$\Sigma_s$ and $\Sigma'_t$ share a single tangent point, 
and that point is a non-degenerate critical point of 
both $f|_{\Sigma'_t}$ and $g|_{\Sigma_s}$, see Figure~\ref{fig:edge_of_graphic}~(i) and (ii); or 
\item
$\Sigma_s$ and $\Sigma'_t$ intersect transversely, and $\Sigma_s \cap \Sigma'_t \cap L$ is a 
single point, see Figure~\ref{fig:edge_of_graphic}~(iii).
\end{itemize}
\begin{figure}[htbp]
\begin{center}
\includegraphics[width=12cm,clip]{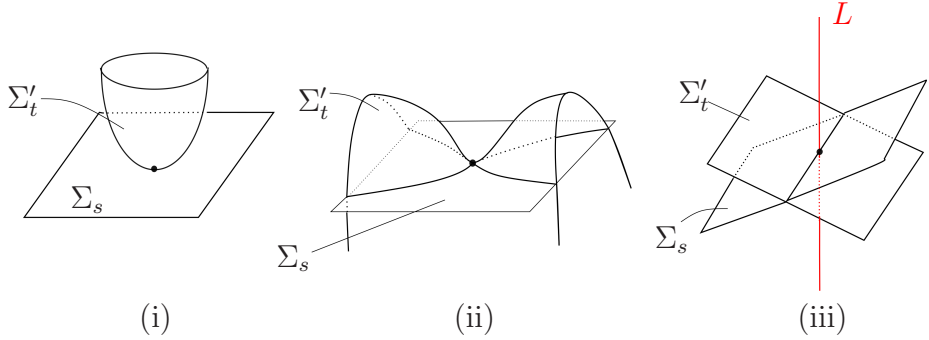}
\begin{picture}(400,0)(0,0)
\put(73,0){(i)}
\put(192,0){(ii)}
\put(322,0){(iii)}
\put(23,83){$\Sigma'_t$}
\put(47,45){$\Sigma_s$}
\put(133,82){$\Sigma'_t$}
\put(125,23){$\Sigma_s$}
\put(277,85){$\Sigma'_t$}
\put(268,30){$\Sigma_s$}
\put(335,115){$\textcolor{red}{L}$}
\end{picture}
\caption{The surfaces $\Sigma_s$ and $\Sigma'_t$ when $(s,t)$ lies in the interior of an edge of the graphic.}
\label{fig:edge_of_graphic}
\end{center}
\end{figure}
If $(s,t)$ is at a $4$-valent vertex of $\Gamma$, then either
\begin{itemize}
\item 
$\Sigma_s$ and $\Sigma'_t$ share exactly two tangent points, and 
those points are non-degenerate critical points of both 
$f|_{\Sigma'_t}$ and $g|_{\Sigma_s}$; 
\item
$\Sigma_s$ and $\Sigma'_t$ share a single tangent point, and that point is a 
non-degenerate critical point of both $f|_{\Sigma'_t}$  and $g|_{\Sigma_s}$. 
Further, $L$ intersects $\Sigma_s \cap \Sigma'_t$ at a point where 
$\Sigma_s$ and $\Sigma'_t$ intersect transversely; or 
\item
$\Sigma_s$ and $\Sigma'_t$ intersect transversely, and $\Sigma_s \cap \Sigma'_t \cap L$ consists of 
exactly two points.
\end{itemize}
If $(s,t)$ is at a $2$-valent vertex of $\Gamma$, then 
$\Sigma_s$ and $\Sigma'_t$ share a single tangent point, and that point is a degenerate critical point 
of both $f|_{\Sigma'_t}$ and $g|_{\Sigma_s}$. See Figure~\ref{fig:2-valent_vertex_of_graphic}. 
\begin{figure}[htbp]
\begin{center}
\includegraphics[width=3.5cm,clip]{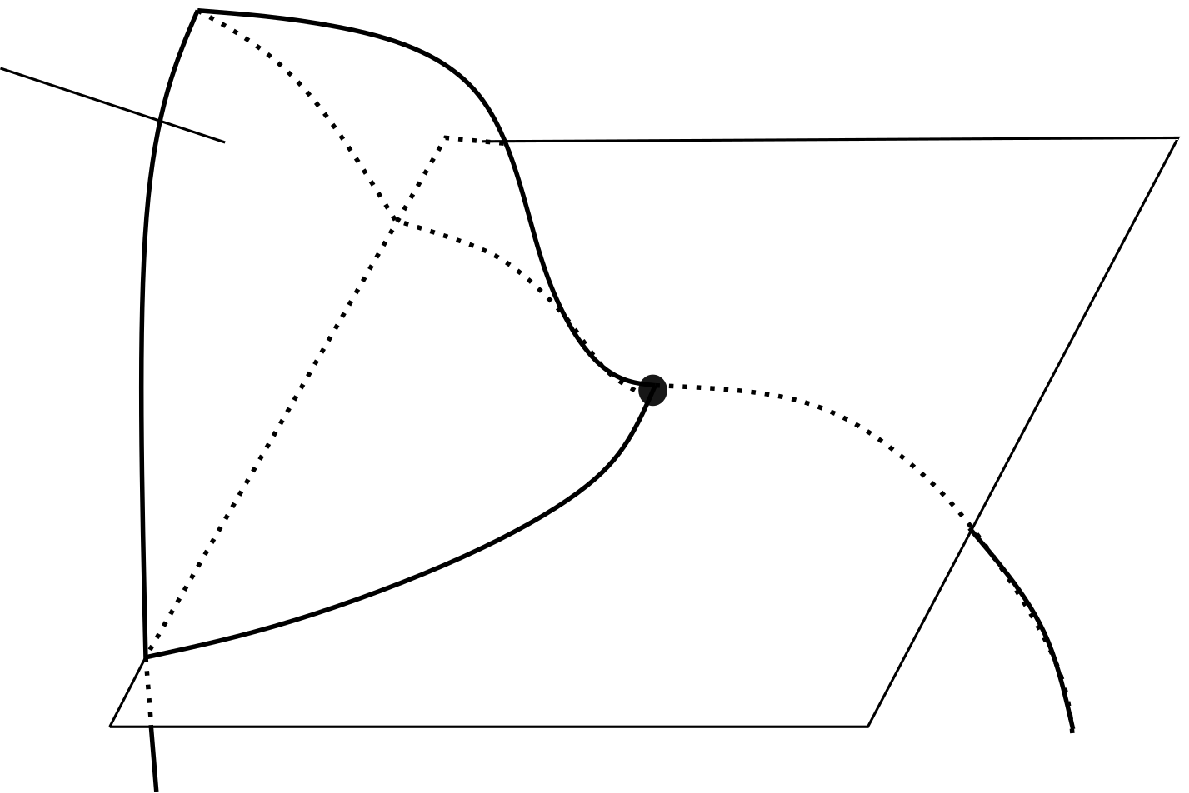}
\begin{picture}(400,0)(0,0)
\put(140,70){$\Sigma'_t$}
\put(210,25){$\Sigma_s$}
\end{picture}
\caption{The surfaces $\Sigma_s$ and $\Sigma'_t$ when $(s,t)$ is at a $2$-valent vertex of the graphic.}
\label{fig:2-valent_vertex_of_graphic}
\end{center}
\end{figure}
Each $1$- or $3$-valent vertex 
is in the boundary of the square, and 
it corresponds to the point 
where the level surface of one of the two sweep-outs is tangent to the spine of the other sweep-out.

\begin{definition}
The graphic defined by $f \times g$ is said to be {\it generic} 
if $f \times g$ is stable in the complement of the spines, and
any vertical or horizontal arc in $[-1,1] \times [-1,1]$ contains at most one vertex of the graphic.  
\end{definition}

The following is Lemma~$34$ of Johnson \cite{Joh10_JTP}.  

\begin{lemma}\label{lemm:generic} 
Let $f$ and $g$ be sweep-outs associated to the bridge decomposition $(M,L;\Sigma)$.    
Let $\{\Phi_{r}:M \rightarrow M\}_{r \in [0,1]}$ be an ambient 
isotopy such that 
$\Phi_{0}=\mathrm{id}_{M}$ and $\Phi_{r}(L)=L$ for all $r \in [0,1]$. 
Set $g_r:=g \circ \Phi_r$ for $r \in [0,1]$. 
Then we can perturb $\{\Phi_{r}\}_{r \in [0,1]}$ slightly, if necessary, so that 
the graphic defined by $f \times g_{r}$ is generic for all but finitely many $r \in [0,1]$. 
At each non-generic $r \in [0,1]$, the graphic fails to be generic due to one of the following two reasons:  
\begin{itemize} 
\item there exists a single vertical or horizontal arc in $[-1,1] \times [-1,1]$ containing 
	two vertices of the graphic, or 
\item the map $f \times g_{r}$ is not stable  in the complement of their spines.  
	$($This case corresponds to the six types of local moves 
	shown in Figure~$5$ of \cite{Joh11}.$)$ 
\end{itemize} 
\end{lemma}

Let $f$ and $g$ be sweep-outs associated to $(M,L;\Sigma)$. 
Let $t,s \in (-1,1)$. 
Set 
$\Sigma_{s}:=f^{-1}(s)$, 
$\Sigma_{t}^{\prime}:=g^{-1}(t)$,
${V_{s}}^{-}:=f^{-1}([-1,s])$, ${V_{s}}^{+}:=f^{-1}([s,1])$, 
${V^{\prime}_{t}}^{-}:=g^{-1}([-1,t])$ and ${V^{\prime}_{t}}^{+}:=g^{-1}([t,1])$. 

\begin{definition}
We say that $\Sigma_{s}$ is {\it mostly above} ({\it mostly below}, respectively) 
$\Sigma_{t}^{\prime}$ if 
each component of 
$\Sigma_{s} \cap {V^{\prime}_{t}}^{-}$ ($\Sigma_{s} \cap {V^{\prime}_{t}}^{+}$, respectively) is
contained in a disk with at most one puncture in $\Sigma_{s} - L$.   
\end{definition}

Let $\mathscr{R}_{a}$ ($\mathscr{R}_{b}$, respectively) 
denote the set of all points $(s,t) \in [-1,1] \times [-1,1]$ such that 
$\Sigma_{s}$ is mostly above (mostly below, respectively) $\Sigma_{t}^{\prime}$. 
The regions $\mathscr{R}_{a}$ and $\mathscr{R}_{b}$ are bounded by the edges of the graphic.   
Note that a point $(s,t)$ near $[-1,1] \times \{-1\}$ is labeled by $\mathscr{R}_{a}$ 
because ${V^{\prime}_{t}}^{-}$ lies within a small neighborhood the spine of $g$, and 
the all intersections of ${V^{\prime}_{t}}^{-}$ and $\Sigma_{s}$ must consist of disks. 
Similarly, a point $(s,t)$ near $[-1,1] \times \{1\}$ is labeled by $\mathscr{R}_{b}$. 
Also, by definition, both regions $\mathrm{Cl}(\mathscr{R}_{a})$ 
and $\mathrm{Cl}(\mathscr{R}_{b})$ are {\it vertically convex},  
that is, 
if a point $(s,t)$ is in $\mathrm{Cl}(\mathscr{R}_{a})$ ($\mathrm{Cl}(\mathscr{R}_{b})$, respectively), 
then so is $(s,t')$ for any $t^{\prime} \le t$ ($t^{\prime} \ge t$, respectively).  

\begin{lemma}\label{lem:intersection_of_regions}
Suppose that $(\mathrm{genus}(\Sigma),n)\not=(0,1),(0,2),(1,1)$.  
Then
the closure of the regions $\mathscr{R}_{a}$ and $\mathscr{R}_{b}$ are disjoint. 
\end{lemma}

\begin{proof}
We first suppose that $\mathscr{R}_{a} \cap \mathscr{R}_{b}\not=\emptyset$. 
Let $(s,t) \in \mathscr{R}_{a} \cap \mathscr{R}_{b}$. 
Then there exists a component $l$ of $\Sigma_{s} \cap \Sigma_{t}^{\prime}$ such that 
$l$ bounds once-punctured disks in $\Sigma_{s}$ in both sides of $l$. 
Thus $\Sigma_{s}$ is a twice-punctured sphere and $(\mathrm{genus}(\Sigma),n)=(0,1)$. 

Next, suppose that $\mathrm{Cl}(\mathscr{R}_{a})$ and $\mathrm{Cl}(\mathscr{R}_{b})$ share an edge of the graphic. 
Let $(s,t)$ be a point in the (interior of the) common edge of 
$\mathrm{Cl}(\mathscr{R}_{a})$ and $\mathrm{Cl}(\mathscr{R}_{b})$. 
A small neighborhood $P$ in $\Sigma_{s}$ of the component of $g|_{\Sigma_{s}}^{-1} (t)$ 
containing a critical point of $g|_{\Sigma_{s}}$ or a point of $L$  
is either a pair of pants or a once-punctured annulus, see Figure~\ref{fig:common_edge}. 
\begin{figure}[htbp]
\begin{center}
\includegraphics[width=12cm,clip]{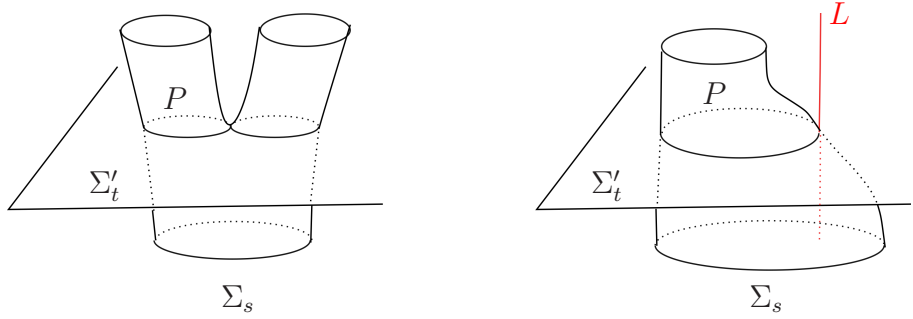}
\begin{picture}(400,0)(0,0)
\put(60,43){$\Sigma'_t$}
\put(110,00){$\Sigma_s$}
\put(250,43){$\Sigma'_t$}
\put(310,00){$\Sigma_s$}
\put(88,75){$P$}
\put(292,77){$P$}

\put(340,107){$\textcolor{red}{L}$}
\end{picture}
\caption{A small neighborhood $P$ in $\Sigma_{s}$ of the component of $g|_{\Sigma_{s}}^{-1} (t)$ 
containing a critical point of $g|_{\Sigma_{s}}$ or a point of $L$.}
\label{fig:common_edge}
\end{center}
\end{figure}
By the assumption, 
each component of $\partial P$ is inessential in 
$\Sigma_{s} - L$. 
Therefore $\Sigma_{s}$ must be a twice-punctured sphere, and thus, 
we have $(\mathrm{genus}(\Sigma),n)=(0,1)$. 

Finally, suppose that $\mathrm{Cl}(\mathscr{R}_{a})$ and $\mathrm{Cl}(\mathscr{R}_{b})$ do not share 
any edge, but they share a vertex of the graphic. 
Let $(s,t_{\pm})$ be points near the vertex shown in Figure~\ref{fig:intersection_of_regions}. 
\begin{figure}[htbp]
\begin{center}
\includegraphics[width=3cm,clip]{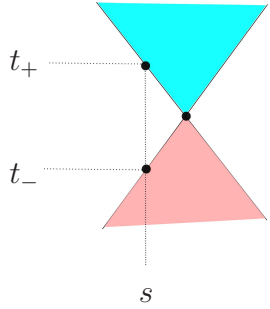}
\begin{picture}(400,0)(0,0)
\put(194,0){$s$}
\put(145,45){$t_{-}$}
\put(145,88){$t_{+}$}
\end{picture}
\caption{A neighborhood of the vertex}
\label{fig:intersection_of_regions}
\end{center}
\end{figure} 
There are exactly two critical points of $g|_{\Sigma_{s}}$ 
between $g|_{\Sigma_{s}}^{-1} (t_{-})$ and $g|_{\Sigma_{s}}^{-1}(t_{+})$: 
one is on $g|_{\Sigma_{s}}^{-1} (t_{-})$ and 
the other is on $g|_{\Sigma_{s}}^{-1} (t_{+})$. 
As in the above case, a small neighborhood $Q_\pm$ in 
$\Sigma_{s}$ of the component of each $g|_{\Sigma_{s}}^{-1}(t_{\pm})$ of concern 
is a pair of pants or a once-punctured annulus. 
In the surface $\Sigma_{s}  - L$, 
each component of $\partial Q_\pm$ either bounds a once-punctured disk or 
cobounds an annulus 
with another component of $\partial Q_\pm$. 
Thus, we can check that $\Sigma_{s}$ is 
either a four-times punctured sphere or a once-punctured torus, which implies
$(\mathrm{genus}(\Sigma),n)=(0,2),(1,1)$. 
\end{proof}

In what follows, we assume that $(\mathrm{genus}(\Sigma),n)\not=(0,1),(0,2),(1,1)$. 
We say that {\it $g$ spans $f$} if 
there there exists $t \in [-1,1]$ such that the horizontal arc $[-1,1] \times \{ t \}$ intersects both 
$\mathscr{R}_{a}$ and $\mathscr{R}_{b}$. 
Otherwise,  we say that {\it $g$ splits $f$}. 
See Figure~\ref{fig:spannig_splitting}. 

\begin{figure}[htbp]
\begin{center}
\includegraphics[width=7cm,clip]{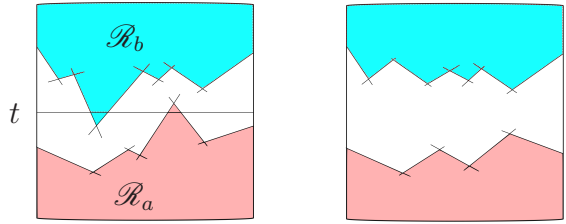}
\begin{picture}(400,0)(0,0)
\put(90,51){$t$}
\put(127,80){$\mathscr{R}_{b}$}
\put(130,20){$\mathscr{R}_{a}$}
\end{picture}
\caption{
The function $g$ spans $f$ 
if some horizontal arc in the square intersects both $\mathscr{R}_{a}$ and $\mathscr{R}_{b}$ (left), 
and otherwise $g$ splits $f$ (right).}
\label{fig:spannig_splitting}
\end{center}
\end{figure}

We say that {\it $g$ spans $f$ positively} 
if there exist points $(a,t) \in \mathscr{R}_{a}$ and $(b,t) \in \mathscr{R}_{b}$ 
with $b<a$.

\begin{lemma}[{\cite[Lemma 14]{Joh10_JTP}}]\label{lem:positively_spanning}
Let $f$ be a sweep-out of $(M, L)$, and $g$ the result of perturbing $f$ slightly  
so that the graphic defined by $f \times g$ is generic. 
Then $g$ spans $f$ positively.   
\end{lemma}

\section{Upper bounds for the distance}\label{sec:upper_bound_for_distance}

Let $(M,L)=(V^+, V^+ \cap L) \cup_{\Sigma} (V^-,V^- \cap L)$ 
be a $(\mathrm{genus}(\Sigma),n)$-decomposition of a link 
$L$ in a closed orientable $3$-manifold $M$, and 
suppose that $(\mathrm{genus}(\Sigma),n)\not=(0,1), (0,2),(1,1)$.  
Recall that  
$d_{\mathcal{PD}}(M,L;\Sigma)=d_{\mathrm{C}(\Sigma_{L})}(\mathcal{PD}(V^+_L),\mathcal{PD}(V^-_L))$.  
The goal of this section is to show the following.  

\begin{theorem}\label{thm:injection}
If $d_{\mathcal{PD}}(M,L;\Sigma) \geq 4$, then
the natural homomorphism $\eta:\mathcal{G}(M,L;\Sigma) \rightarrow \mathrm{MCG}(M,L)$ is injective.  
\end{theorem}

We prove Theorem~\ref{thm:injection}. 
We first note that, by Lemma \ref{lem:Haken}, we may assume the following. 

\vspace{0.5em}
\noindent
\textbf{Assumption:}~Any meridional loop of $L$ does not bound a disk in $M-L$. 

\begin{lemma}
\label{lem:Euler characteristic after a compression}
Let $L$ and $M$ be as above.  
Let $\Sigma$ be a closed connected surface in $M$ intersecting $L$ transversely. 
Let $D$ be a disk in $M$ such that $D \cap \Sigma = \partial D$, 
$\partial D \cap L = \emptyset$,  and 
$D$ intersects $L$ transversely in at most one point. 
Let $\Sigma'$ be a component of a surface obtained by compressing $\Sigma$ along $D$. 
Then we have $\chi (\Sigma' - L) \geq \chi (\Sigma - L)$,  
where $\chi(\cdot)$ denotes the Euler characteristic. 
\end{lemma}
\begin{remark}
In the above lemma, we allow the case where $D$ is not a compression disk for 
$\Sigma$, in other words, 
$\partial D$ can be inessential in $\Sigma$.  
\end{remark}
\begin{proof}
Suppose that $\chi (\Sigma' - L) < \chi (\Sigma - L)$. 
Then the only possibility is that $|D \cap L| = 1$, 
$\partial D$ bounds a disk $E$ in $\Sigma$ with $E \cap L = \emptyset$, and 
$\Sigma' = (\Sigma - E) \cup D$. 
This contradicts our assumption stated right before the lemma. 
\end{proof}

\begin{lemma}\label{lem:level_surface}
Let $\Sigma$ be a closed orientable surface, 
$K$ the union of vertical arcs in $\Sigma \times [0,1]$, and 
$S$ a surface in $\Sigma \times [0,1]$ that intersects $K$ transversely. 
If $S$ separates $\Sigma \times \{0\}$ from $\Sigma \times \{1\}$, 
then $\chi(S_{K}) \le \chi(\Sigma_{K})$. 
Furthermore, the equality holds if and only if 
$S$ is isotopic to a horizontal surface keeping $S$ transverse to $K$ throughout the isotopy. 
\end{lemma}

\begin{proof}
Let $S^{\prime}$ be the result of repeatedly compressing $S_{K}$ 
so that $S^{\prime}$ is incompressible in $(\Sigma \times [0,1]) - K$. 
The surface $S^{\prime}$ still separates $\Sigma \times \{0\}$ from $\Sigma \times \{1\}$, 
and it follows from Lemma~\ref{lem:Euler characteristic after a compression} that $\chi(S_{K}) \le \chi(S^{\prime})$. 
Since any incompressible surface in $\Sigma_{K} \times [0,1]$ is isotopic to a horizontal surface, 
we have $\chi(S_{K}) \le \chi(\Sigma_{K})$. 
\end{proof}

Let $f:M \rightarrow [-1,1]$ be a sweep-out of $(M, L)$ with $f^{-1}(0)=\Sigma$, and 
$g$ the result of perturbing  $f$ slightly. 
Let $[\phi]$ be in the kernel of  $\eta$. 
Then, there exists an ambient isotopy 
$\{\Phi_{r}:M \rightarrow M\}_{r \in [0,1]}$ such that  
$\Phi_{0}=\mathrm{id}_{M}$, 
$\Phi_{1}=\phi$, 
and $\Phi_{r}(L)=L$ for all $r \in [0,1]$. 
We can assume that $\{\Phi_{r}\}_{r \in [0,1]}$ satisfies 
the conditions described in Lemma~\ref{lemm:generic}, that is, 
only a finitely many element in the $1$-parameter family 
$\{ g_{r}:=g \circ \Phi_{r}\}_{r \in [0,1]}$ of sweep-outs of $(M,L)$
is non-generic.
\begin{lemma}\label{lem:span}
If $g_{r}$ spans $f$ for all $r \in [0,1]$, 
then $\phi|_{\Sigma}$ is isotopic in $\Sigma$ to the identity $\mathrm{id}|_{\Sigma}$ 
relative to the points $\Sigma \cap L$.  
\end{lemma}

\begin{proof}
For each $r \in [-1,1]$, set $A_{r}:=p_{2}(\mathrm{Cl}(\mathscr{R}_{a}))$ 
and $B_{r}:=p_{2}(\mathrm{Cl}(\mathscr{R}_{b}))$, 
where $p_{2}:[-1,1] \times [-1,1] \rightarrow [-1,1]$ denotes the projection onto 
the second coordinate. 
Since $g_r$ spans $f$, $A_{r}$ and $B_{r}$ have non-empty intersection. 
Indeed,  $A_{r} \cap B_{r}$ is a closed interval in $[-1,1]$ 
because $\mathrm{Cl}(\mathscr{R}_{a} )$ and $\mathrm{Cl}(\mathscr{R}_{b})$ 
are vertically convex subsets of $[-1,1] \times [-1,1]$.   
Fix $r \in [-1,1]$.
We define the map $\varphi_r$ from the surface $g^{-1}(0)$ to $f^{-1}(0)$ 
that sends the points $g^{-1}(0) \cap L$ to $f^{-1}(0) \cap L$
as follows.  

Let $t (r)$ be an interior point of the closed interval $A_{r} \cap B_{r}$. 
There are points $a(r)$ and  $b(r)$ in $[-1,1]$ such that 
$(a(r), t (r)) \in \mathscr{R}_{a}$ and  $(b(r), t (r)) \in \mathscr{R}_{b}$, respectively. 
Set $\Sigma_{a(r)}:= f^{-1} (a(r))$, $\Sigma_{b(r)} := f^{-1} (b(r))$ and $\Sigma'_{t(r)} := g^{-1} (t(r))$. 
Since $\Sigma_{a(r)}$ is mostly above $\Sigma'_{t (r)}=g_{r}^{-1}(t (r))$ 
while $\Sigma_{b(r)}$ is mostly below $\Sigma'_{t (r)}$, 
we obtain a surface $S_r$ lying within the product region between $\Sigma_{a(r)}$ and $\Sigma_{b(r)}$  
by repeatedly compressing $\Sigma'_{t (r)}$ 
along the innermost disks intersecting $L$ at most once in $\Sigma_{a(r)} \cup \Sigma_{b(r)}$ as long as possible.
By the construction,  
the surface $S_{r}$ separates $\Sigma_{a(r)}$ from $\Sigma_{b(r)}$. 
See Figure~\ref{fig:compressing}. 

\begin{figure}[htbp]
\begin{center}
\includegraphics[width=10cm,clip]{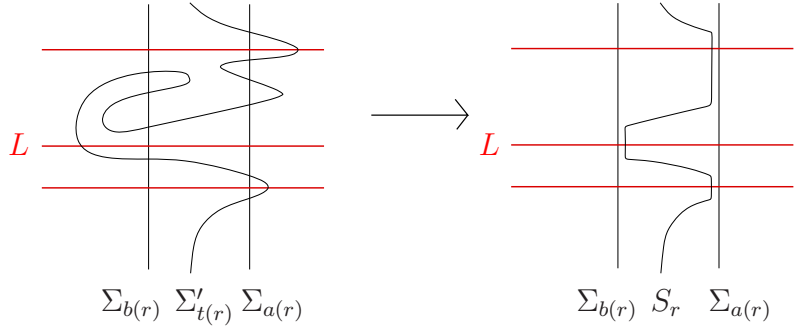}
\begin{picture}(400,0)(0,0)
\put(133,0){$\Sigma_{a(r)}$}
\put(80,0){$\Sigma_{b(r)}$}
\put(108,0){$\Sigma'_{t(r)}$}
\put(45,59){\color{red}$L$}
\put(310,0){$\Sigma_{a(r)}$}
\put(260,0){$\Sigma_{b(r)}$}
\put(223,59){\color{red}$L$}
\put(288,0){$S_{r}$}
\end{picture}
\caption{Compressing along innermost simple closed curves in $\Sigma_{a(r)} \cup \Sigma_{b(r)}$ iterativlely to form the surface $S_{r}$.}
\label{fig:compressing}
\end{center}
\end{figure}

We argue that $S_r$ is canonically isotopic to a level surface of the sweep-out $f$ 
keeping $S_{r}$ transverse to the link $L$ throughout the isotopy. 
By Lemma~\ref{lem:Euler characteristic after a compression}
we have $\chi(S_{r} - L) \geq \chi (\Sigma'_{t(r)} - L )$. 
On the other hand, 
since $S_{r}$ separates $\Sigma_{a(r)}$ from $\Sigma_{b(r)}$, 
we have $\chi(S_{r} - L) \le \chi(\Sigma'_{t (r)} - L)$ 
by Lemma~\ref{lem:level_surface}. 
Thus,  
we have  $\chi(S_{r} - L) = \chi(\Sigma'_{t (r)} - L)$. 
Again, by Lemma~\ref{lem:level_surface}, 
$S_r$ is isotopic to a level surface of the sweep-out $f$ 
with keeping the surfaces transverse to $L$ throughout. 

By the argument above, 
it follows that $S_r$ coincides with 
$\Sigma'_{t (r)}$ away from some disks, each of which intersects $L$ at most once. 
Thus, there is a canonical identification of $S_r$ with $\Sigma'_{t (r)}$. 
Therefore, we have the following map: 
$$g^{-1}(0) \rightarrow g_{r}^{-1}(t (r)) = 
\Sigma_{t (r)}^{\prime} \rightarrow S_{r} \rightarrow f^{-1}(0).$$ 
Note that all maps are uniquely defined up to isotopy 
(with fixing the intersection points between the surface of concern and $L$).
It is clear that the composition map can be chosen 
so that it sends $g^{-1}(0) \cap L$ to $f^{-1}(0) \cap L$. 
Define the map $\varphi_{r} : g^{-1}(0) \to f^{-1}(0)$ by such a composition map.   

We shall now show that $\phi|_{\Sigma}$ is isotopic to the identity relative to $\Sigma \cap L$.  
There is the canonical identification of $f^{-1}(0) = \Sigma$ with 
$g^{-1}(0)$ because 
$g$ is the result of perturbing $f$ slightly. 
Under this identification, 
it holds that $\varphi_{0}=\mathrm{id}_{\Sigma}$ and $\varphi_{1}=\phi|_{\Sigma}$.  
It is clear that the values of $t (r)$ can be chosen  
so that it varies continuously. 
Although perhaps the points $a(r)$ and $b(r)$ 
do not vary continuously at some finitely many values of $r$, 
the deformation of $\varphi_r$ remains to be continuous even around such values: 
it is easily seen that 
the choice of $a(r)$ or $b(r)$ does not affect 
the definition of the map $g_{r}^{-1}(t (r)) \rightarrow f^{-1}(0)$ 
in the above argument modulo isotopy. 
Thus, we conclude 
that for any $r, r' \in [0,1]$, $\varphi_r$  and $\varphi_{r'}$ are isotopic 
fixing $\Sigma \cap L$, which shows the proof. 
\end{proof}

\begin{lemma}\label{lem:split}
If there exists $r \in [0,1]$ such that $g_{r}$ splits $f$, 
then $d_{\mathcal{PD}}(M,L;\Sigma) \le 3$. 
\end{lemma}

\begin{proof}
We denote by $\pi_0$ the natural projection from the preimage 
$f^{-1} ((-1, 1))$ of the open interval $(-1,1)$ to $\Sigma$ that 
maps $f^{-1} ((-1, 1)) \cap L$ to $\Sigma \cap L$. 
By Lemma~\ref{lem:positively_spanning}, $g_{0}$ spans $f$ positively. 
Thus, there exists a time $r_0$ such that 
\begin{itemize}
\item $g_{r}$ spans $f$ positively for all $r<r_{0}$, and 
\item $A_{r_{0}} \cap B_{r_{0}}=\{t\}$.  
\end{itemize}
In the following, we consider the graphic defined by $f \times g_{r_0}$. 
By Lemma \ref{lemm:generic}, 
the arc $[-1,1] \times \{t\}$ must intersect the region $\mathrm{Cl}(\mathscr{R}_a) \cup \mathrm{Cl}(\mathscr{R}_b)$ 
in exactly two vertices of the graphic.  
Let $(a,t ) \in \mathrm{Cl}(\mathscr{R}_{a})$ and $(b,t ) \in \mathrm{Cl}(\mathscr{R}_{b})$ 
be coordinates of such vertices. 
We note that $b < a$. 
Let us consider the points near these vertices shown in 
Figure~\ref{fig:nbd_of_vertices}:  
their coordinates are $(a_{-}, t )$, $(b_{+},t )$, 
$(a_{\pm},t _{\pm})$ and $(b_{\pm},t _{\pm})$.

\begin{figure}[htbp]
\begin{center}
\includegraphics[width=6cm,clip]{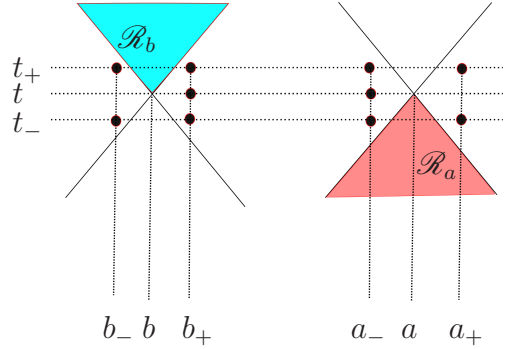}
\begin{picture}(400,0)(0,0)
\put(140,110){$\mathscr{R}_{b}$}
\put(253,62){$\mathscr{R}_{a}$}
\put(247,0){$a$}
\put(228,0){$a_{-}$}
\put(265,0){$a_{+}$}
\put(149,0){$b$}
\put(134,0){$b_{-}$}
\put(164,0){$b_{+}$} 
\put(100,89){$t$}
\put(100,78){$t_{-}$}
\put(100,99){$t_{+}$}
\end{picture}
\caption{Small perturbed points of the vertices.}
\label{fig:nbd_of_vertices}
\end{center}
\end{figure}

We set
$\Sigma_{s}:=f^{-1}(s)$, 
$\Sigma_{t}^{\prime}:=g_{r_0}^{-1}(t)$ as before. 
Set 
$f_{0}:=f|_{\Sigma'_{t}}$ and $f_{\pm}:=f|_{\Sigma'_{t_{\pm}}}$.  
Note that the functions $f_{\pm}$ are Morse away from the preimages of $\pm1$. 

We think about the function $f_{0}$. 
Let $\mathcal{L}_a$ be the set of simple closed curves of $f_{0}^{-1}(a_-)$ 
that are essential in $\Sigma_{a_-}-L$.  
Similarly, let $\mathcal{L}_b$ be the set of simple closed curves of $f_{0}^{-1}(b_+)$ 
that are essential in $\Sigma_{b_+}-L$. 
We note that $\mathcal{L}_a \not= \emptyset$ ($\mathcal{L}_b \not= \emptyset$, respectively) 
because $(a_-,t)$ ($(b_+,t)$, respectively) does not lie in $\mathscr{R}_{a} \cup \mathscr{R}_{b}$. 
Let $l_a$ and $l_b$ be arbitrary simple closed curves 
in $\mathcal{L}_a$ and $\mathcal{L}_b$, respectively. 

By the choice of $r_0$, 
the same argument of the proof of Lemma~\ref{lem:span} 
shows that we can find a natural map  
$\rho_0 : f_0^{-1}([b_+ , a_- ])  \to \Sigma_0 = \Sigma$ 
that extends to a homeomorphism $\hat{\rho}_0$ from the whole surface $\Sigma'_{t}$ to $\Sigma$ with $\hat{\rho}_0(\Sigma'_{t} \cap L) = \Sigma \cap L$.  
Since both $l_{a}$ and $l_{b}$ are level loops of $f_{0} : \Sigma'_t \to (-1,1)$, 
they are disjoint. 
Therefore, the images $\rho_0(l_a)$ and $\rho_0(l_b)$ in $\Sigma_{0}=\Sigma$ 
are also disjoint. 
In other words, we have 
\begin{eqnarray}\label{eq:1}
d_{\mathcal{C}(\Sigma_{L})} (\rho_0(l_{b}), \rho_0(l_{a}) ) \le 1, 
\end{eqnarray}
where 
we regard $\rho_0(l_{a})$ and $\rho_0(l_{b})$ as vertices of $\mathcal{C}(\Sigma_{L})$. 
The projection $\pi_0$ also takes $l_a$ and $l_b$ to simple closed curves in $\Sigma$, which may have non-empty 
intersection. 
However, we see from the definition of $\rho_0$ that $\pi_0(l_a)$ and $\rho_0 (l_a)$ 
($\pi_0(l_a)$ and $\rho_0 (l_a)$, respectively) are homotopic, hence, isotopic. 
Therefore, if we regard $\pi_0(l_{a})$ and $\pi_0(l_{b})$ as vertices of $\mathcal{C}(\Sigma_{L})$, we can write 
\begin{eqnarray}\label{eq:1'}
d_{\mathcal{C}(\Sigma_{L})} (\pi_0(l_{b}), \pi_0(l_{a}) ) \le 1. 
\end{eqnarray}

We next show the following inequality: 

\begin{eqnarray}\label{eq:2}
d_{\mathcal{C}(\Sigma_{L})}(\pi_0 (\mathcal{L}_a),\mathcal{PD}(V^{+}_{L})) \le 1. 
\end{eqnarray} 

We note that 
any level loop of $f^{-1}_{0}(a_-)$ can also be regard as loops of each of $f^{-1}_{\pm}(a_-)$ 
since the points $(a_{-}, t)$ and $(a_{-},t_{\pm})$ are in the same component of the complementary region of the graphic. 
In the following, for simplicity, we shall not distinguish between 
a level loop of $f^{-1}_{0}(a_-)$ and the corresponding loops of $f^{-1}_{\pm}(a_-)$ 
in their notations.    
Let $l_a \in \mathcal{L}_a$. 
Let us first consider the function $f_{-}$.  
Since the point $(a,t_{-})$ lies within $\mathscr{R}_{a}$, 
as we pass from the level $a_-$ to the level $a$, the simple closed curve $l_a$ turns into one or two 
inessential simple closed curves in 
$\Sigma_a - L$. 
Therefore, in the surface $\Sigma_{a_-} - L$, either 
\begin{itemize}
\item
The simple closed curve $l_a$ bounds a twice-punctured disk, see Figure~\ref{fig:level_set_f_-}~(i) and (ii); or 
\item
The simple closed curve $l_a$ cobounds with another essential simple closed curve $l'_a$ an annulus that 
intersects $L$ in at most one point, see Figure~\ref{fig:level_set_f_-}~(iii),
\end{itemize} 
and the other simple closed curves of $f_-^{-1}(a_-)$ are inessential in 
$\Sigma_{a_-} - L$. 
\begin{figure}[htbp]
\begin{center}
\includegraphics[width=14cm,clip]{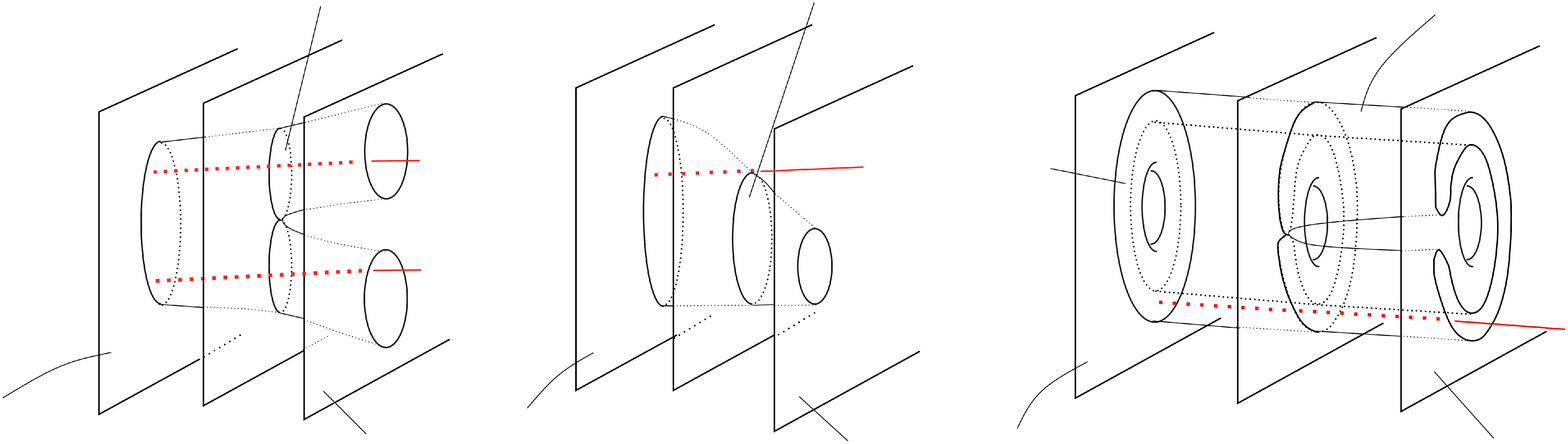}
\begin{picture}(400,0)(0,0)
\put(55,0){(i)}
\put(175,0){(ii)}
\put(325,0){(iii)}
\put(30,50){$l_a$}
\put(158,50){$l_a$}
\put(278,50){$l_a$}
\put(261,81){$l'_a$}
\put(-10,20){$\Sigma_{a_-}$}
\put(130,10){$\Sigma_{a_-}$}
\put(255,5){$\Sigma_{a_-}$} 
\put(90,5){$\Sigma_a$}
\put(220,5){$\Sigma_a$}
\put(380,5){$\Sigma_a$} 
\put(76,128){$\Sigma'_{t_-}$}
\put(207,128){$\Sigma'_{t_-}$}
\put(367,125){$\Sigma'_{t_-}$}
\end{picture}
\caption{Potential configurations of $l_a$ in $\Sigma'_{t_-}$.}
\label{fig:level_set_f_-}
\end{center}
\end{figure}
As explained above, 
the natural map $\rho_0:f_0^{-1}([b_+ , a_- ]) \rightarrow \Sigma_0=\Sigma$ 
can be extended to the map $\hat{\rho}_0$ defined on the whole surface. 
The same thing still holds for the natural map from $f_-^{-1}([b_+ , a_- ])$ to $\Sigma$. 
Due to the existence of an extension of the natural map 
we see that, in the surface $\Sigma'_{t_-} - L$, either 
\begin{itemize}
\item
The simple closed curve $l_a$ bounds a twice-punctured disk; or 
\item
The simple closed curve $l_a$ cobounds with another simple closed curve 
$l'_a$ an annulus that intersects $L$ in at most one point,
\end{itemize} 
according to which of the above two cases of the configuration of $l_a$ in $\Sigma_{a_-} - L$ occurs. 
The other simple closed curves of $f^{-1}_{-}(a_{-})$  are inessential even in $\Sigma'_{t_-} - L$.

Let us next consider the function $f_{+}$. 
Recall that we denote the simple closed curve in 
$\Sigma'_{t_+}$ corresponding to $l_a \subset \Sigma'_{t}$ by the same symbol $l_a$. 

\noindent
\textbf{Case~A:} The simple closed curve $l_a$ bounds a twice-punctured disk 
in $\Sigma_{t_+}-L$ (and hence in $\Sigma_{a_-}-L$). 

Let $P$ be the twice-punctured disk in $\Sigma'_{t_+}-L$ bounded by $l_a$ 
(note that such a subsurface is unique 
because $(\mathrm{genus}(\Sigma),n) \not= (0,2)$). 
We note that, in this case, $\mathcal{L}_a=\{l_a\}$.  
As we pass from the level $a_-$ to the level $a$, there are the following four cases to consider. 

\noindent 
\textbf{Case~A1:}  
One or two new simple closed curves are created away from $l_a$ (Figure \ref{fig:level_set_f_+_case1}).  
\begin{figure}[htbp]
\begin{center}
\includegraphics[width=4cm,clip]{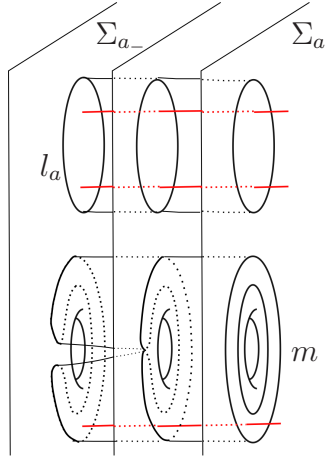}
\begin{picture}(400,0)(0,0)
\put(250,50){$m$}
\put(155,120){$l_a$}
\put(176,170){$\Sigma_{a_-}$}
\put(250,170){$\Sigma_a$}
\end{picture}
\caption{The simple closed curve $m \subset f_+^{-1} (a)$ in $\Sigma'_{t_+}$ in Case~A1. }
\label{fig:level_set_f_+_case1}
\end{center}
\end{figure}

We first see that at least one of the two new simple closed curves is essential in $\Sigma_a-L$. 
Suppose, contrary to our claim, that both of them are inessential in  $\Sigma_a-L$. 
As we pass from the level $a$ to the level $a_+$, 
the simple closed curve $l_{a}$ turns into one or two inessential simple closed curves in $\Sigma_{a_+}-L$. 
Thus, all of the simple closed curves in $f_+^{-1}(a_+)$ are inessential in $\Sigma_{a_+}-L$. 
However, this contradicts the fact 
that the point $(a_+,t_+)$ lies in the complement in $[-1,1] \times [-1,1]$ of $\mathscr{R}_{a} \cup \mathscr{R}_{b}$. 

Let $m$ be one of the new simple closed curves that is essential in $\Sigma_a-L$.  
Since each simple closed curve of $f^{-1}_+(a_{-})$ is inessential in $\Sigma'_{t_+}$ except for $l_a$, 
the curve $m$ is contained in a disk with at most one puncture 
in $\Sigma'_t - L$. 
Thus, $m$ is also inessential in $\Sigma'_{t_+} - L$. 
Let $D \subset \Sigma'_{t_+} - L$ be a disk with at most one puncture  bounded by $m$. 
By repeatedly compressing $D$ 
along the innermost disk with at most one puncture in $\Sigma_a-L$ as long as possible, 
we finally obtain a disk $D'$ in the handlebody $V^+_{a} = f^{-1} ([a, 1])$ such that 
$\partial D' = m$ and $|D' \cap L| \leq 1 $. 
Thus, $\pi_0(m)$ is a vertex of $\mathcal{PD} (V^+_L)$. 
As shown in Figure \ref{fig:level_set_f_+_case1}, 
the inequality $d_{\mathcal{C}(\Sigma_{L})}(\pi_0(l_{a}),\pi_0(m)) \le 1$ holds. 
In consequence, we have 
$d_{\mathcal{C}(\Sigma_{L})}(\pi_0(l_{a}),\mathcal{PD}(V^{+}_{L})) \le 1$. 

\noindent 
\textbf{Case~A2:} The simple closed curve $l_a$ and another simple closed curve $c$ in $f_+^{-1}(a_-)$ 
are pinched together to produce a new simple closed curve $m$ (Figure \ref{fig:configuration_of_curves_A2}).  

Since the point $(a,t_+)$ is in the complement in $[-1,1] \times [-1,1]$ of $\mathscr{R}_{a} \cup \mathscr{R}_{b}$, 
the simple closed curve $m$ is essential in $\Sigma_a-L$. 
We also see that $c$ bounds a once-punctured disk in $\Sigma_{a_-}-L$. 
Suppose, contrary to our claim, that $c$ bounds a disk in $\Sigma_{a_-}-L$.  
Hence $\pi_0(m)$ is isotopic to $\pi_0(l_a)$ in $\Sigma-L$. 
As we pass from the level $a$ to the level $a_+$, 
the simple closed curve $m$ turns into one or two inessential simple closed curves in $\Sigma_{a_+}-L$, 
but this is impossible because the point $(a_+,t_+)$ lies in the complement 
of $\mathscr{R}_a \cup \mathscr{R}_b$. 

By the assumption, any meridional loop of $L$ does not bound a disk in $M-L$. 
Thus, $c$ bounds no disk in $\Sigma_{t_+}-L$. 
The possible configuration in $P$ of $l_a$, $m$ and $c$ is shown in Figure \ref{fig:configuration_of_curves_A2}. 
In particular, $m$ bounds a once-punctured disk $D$ in $P-L$. 
By repeatedly compressing $D$ 
along the innermost disk with at most one puncture in $\Sigma_{a} -L$ as long as possible, we finally obtain a disk $D'$ in the handlebody $V^+_{a}$ such that 
$\partial D' = m$ and $|D' \cap L| = 1 $. 
As the points $(a_-,t_+)$ and $(a,t_+)$ can be connected by a path that intersects the graphic once, 
$d_{\mathcal{C}(\Sigma_{L})}(\pi_0(l_{a}),\pi_0(m)) \le 1$ holds. 
Therefore, it follows that $d_{\mathcal{C}(\Sigma_{L})}(\pi_0(l_{a}),\mathcal{PD}(V^{+}_{L})) \le 1$. 
\begin{figure}[htbp]
\begin{center}
\includegraphics[width=8.5cm,clip]{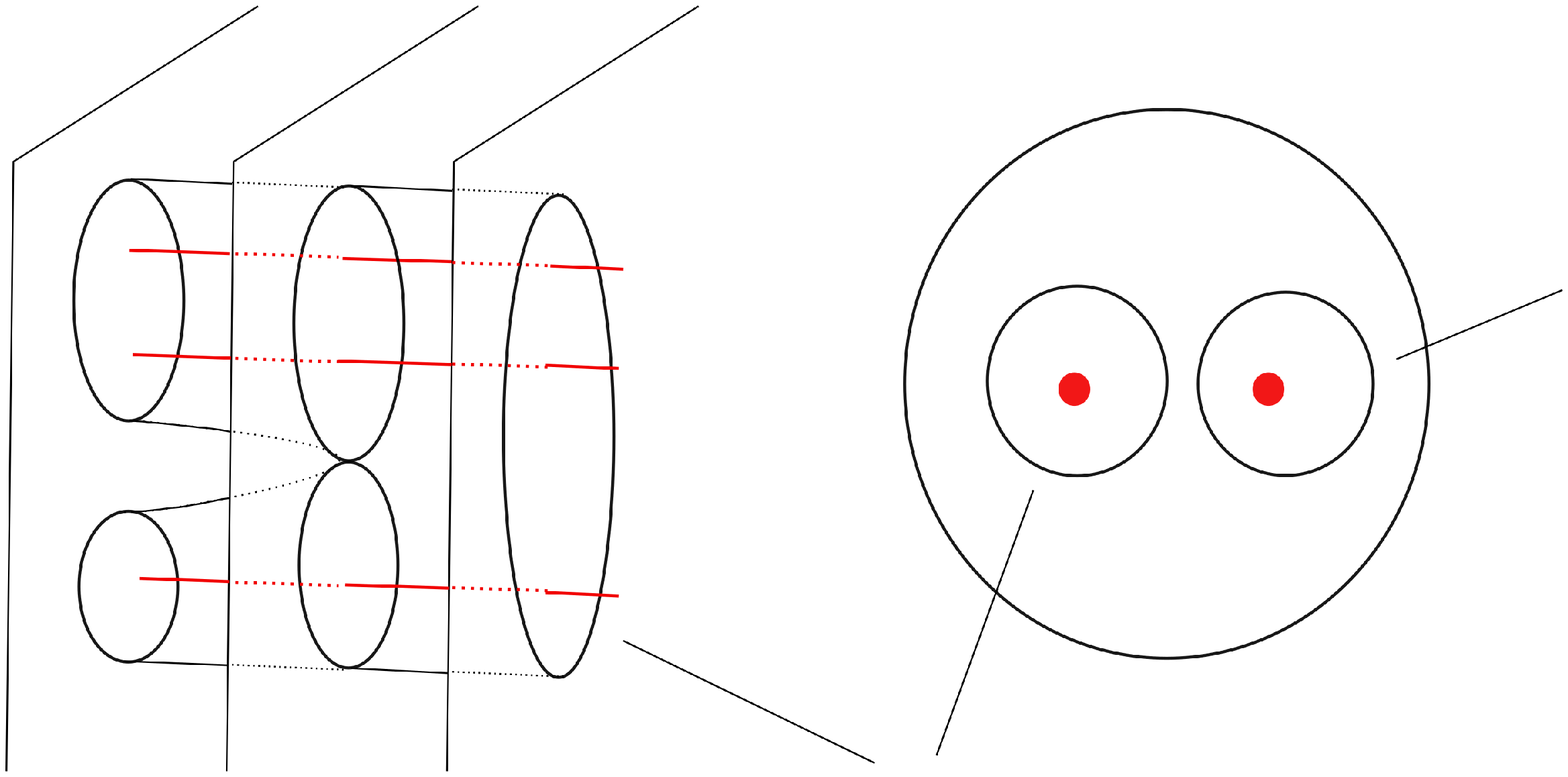}
\begin{picture}(400,0)(0,0)
\put(83,70){$l_a$}
\put(208,70){$l_a$}
\put(214,8){$m$}
\put(84,30){$c$}
\put(322,86){$c$}
\put(255,40){$P$}
\put(108,117){$\Sigma_{a_-}$}
\put(177,117){$\Sigma_a$}
\end{picture}
\caption{Case~A2.}
\label{fig:configuration_of_curves_A2}
\end{center}
\end{figure}

\noindent 
\textbf{Case~A3:} The simple closed curve $l_a$ passes through a puncture and turns into a new simple closed curve $m$ 
(Figure \ref{fig:configuration_of_curves_A3}). 

Since the point $(a,t_+)$ is in the complement in $[-1,1] \times [-1,1]$ of $\mathscr{R}_{a} \cup \mathscr{R}_{b}$, 
$m$ is essential in $\Sigma_a-L$. 
As shown in Figure \ref{fig:configuration_of_curves_A3}, 
$m$ bounds a once-punctured disk $D$ in $P-L$. 
By repeatedly compressing $D$ 
along the innermost disk with at most one puncture in $\Sigma_a -L$ as long as possible, we finally obtain a disk $D'$ in the handlebody $V^+_{a}$ such that 
$\partial D' = m$ and $|D' \cap L| = 1 $. 
As $d_{\mathcal{C}(\Sigma_{L})}(\pi_0(l_{a})),\pi_0(m)) \le 1$, 
it follows that $d_{\mathcal{C}(\Sigma_{L})}(\pi_0(l_{a}),\mathcal{PD}(V^{+}_{L})) \le 1$. 
\begin{figure}[htbp]
\begin{center}
\includegraphics[width=8cm,clip]{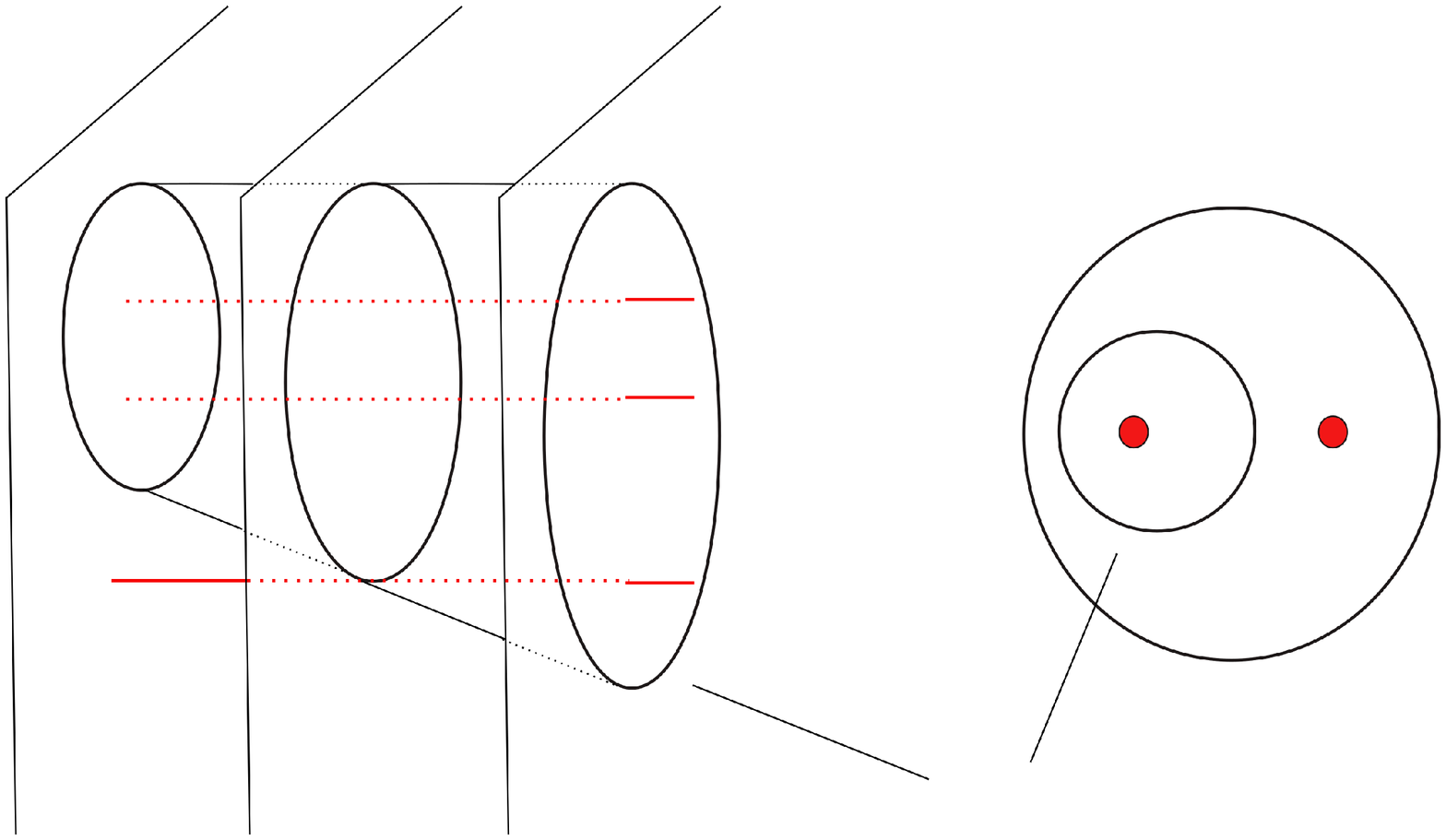}
\begin{picture}(400,0)(0,0)
\put(90,70){$l_a$}
\put(235,75){$l_a$}
\put(237,15){$m$}
\put(280,45){$P$}
\put(115,126){$\Sigma_{a_-}$}
\put(190,124){$\Sigma_a$}
\end{picture}
\caption{Case ~A3.}
\label{fig:configuration_of_curves_A3}
\end{center}
\end{figure}

\noindent 
\textbf{Case~A4:} The simple closed curve $l_a$ is pinched to produce 
two simple closed curves $m_1$ and $m_2$ 
(Figure \ref{fig:configuration_of_curves_A4}).  

There are two possible configurations of $l_a$, $m_1$ and $m_2$ in $P$.  
See Figure \ref{fig:configuration_of_curves_A4}. 
\begin{figure}[htbp]
\begin{center}
\includegraphics[width=8cm,clip]{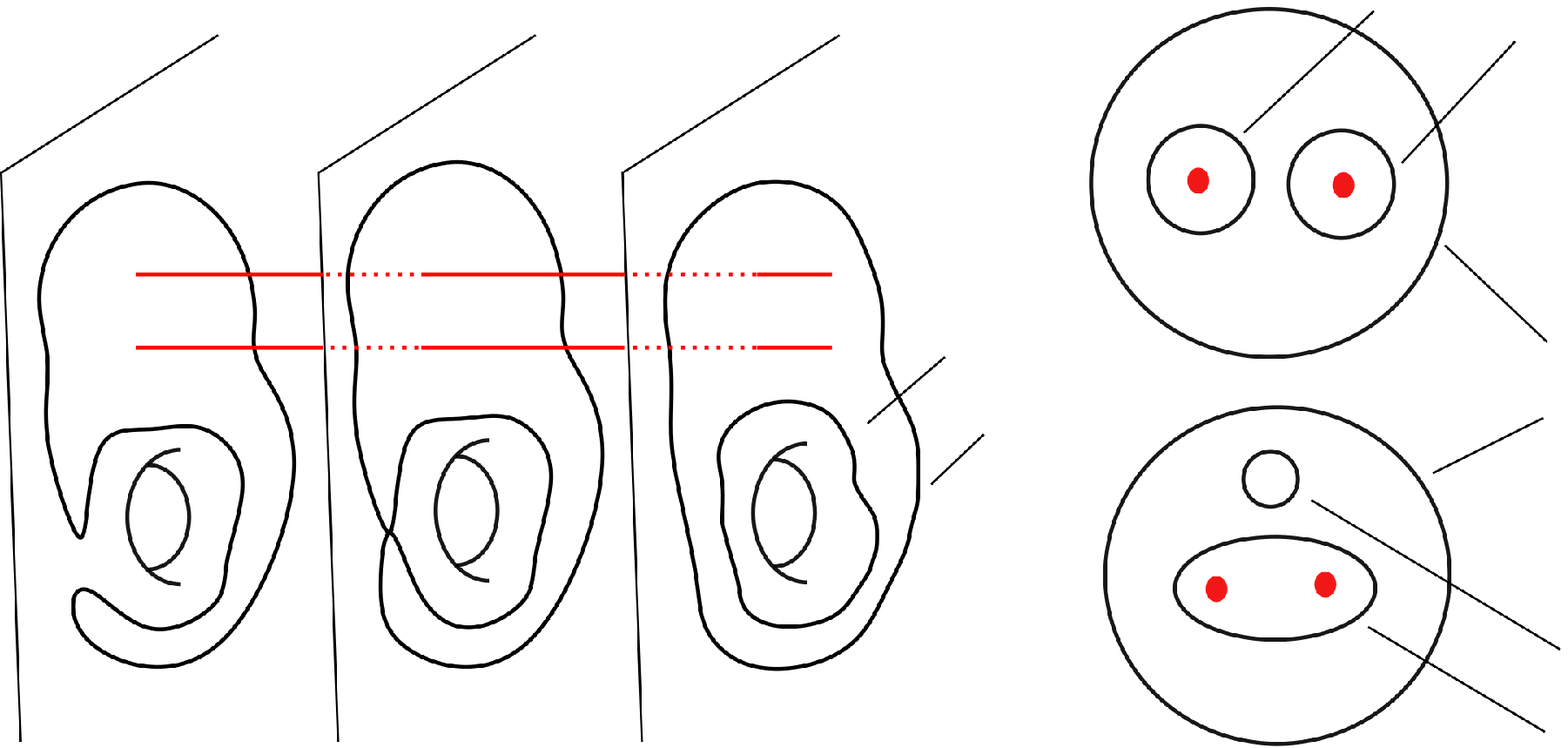}
\begin{picture}(400,0)(0,0)
\put(95,62){$l_a$}
\put(225,72){$m_1$}
\put(230,60){$m_2$}
\put(317,20){$m_1$}
\put(310,10){$m_2$}
\put(287,122){$m_1$}
\put(310,117){$m_2$}
\put(315,65){$l_a$}
\put(262,17){$P$}
\put(262,75){$P$}
\put(115,106){$\Sigma_{a_-}$}
\put(203,104){$\Sigma_a$}
\end{picture}
\caption{Case~A4.}
\label{fig:configuration_of_curves_A4}
\end{center}
\end{figure}

First, suppose that both of the two simple closed curves $m_1$ and $m_2$ bound once-punctured disks, 
which are mutually disjoint, in $P-L$. 
Since the point $(a,t_+)$ is in the complement in $[-1,1] \times [-1,1]$ of $\mathscr{R}_{a} \cup \mathscr{R}_{b}$, 
one of $m_1$ and $m_2$ is essential in $\Sigma_a-L$. 
We may assume that $m_1$ is essential in $\Sigma_a-L$. 
Let $D \subset P$ be the disk such that $\partial D=m_1$ and $|D \cap L|=1$.  
By repeatedly compressing $D$ 
along the innermost disk with at most one puncture in $\Sigma_a -L$ as long as possible, 
we finally obtain a disk $D'$ in the handlebody $V^+_{a}$ such that 
$\partial D' = m_1$ and $|D' \cap L|=1$. 
Thus, we have 
$d_{\mathcal{C}(\Sigma_{L})}(\pi_0(l_{a}),\mathcal{PD}(V_L^+))\le d_{\mathcal{C}(\Sigma_{L})}(\pi_0(l_{a}),\pi_0(m_1)) \le 1$. 

Next, suppose that $m_1$ bounds a disk in $P-L$. 
It suffices to show that $m_1$ must be essential in $\Sigma_a-L$. 
Indeed, if $m_1$ is essential in $\Sigma_a-L$, 
a similar argument as above shows that  
there exists a disk $D$ in $V_a^+-L$ such that $\partial D=m_1$.  
Thus, we have $d_{\mathcal{C}(\Sigma_{L})}(\pi_0(l_{a}),\mathcal{PD}(V_L^+))\le 1$. 

Suppose, for the sake of contradiction, 
$m_1$ is inessential in $\Sigma_a-L$. 
By the assumption, any meridional loop of $L$ does not bound a disk in $M-L$. 
Hence, $m_1$ must bound a disk in $\Sigma_a-L$. 
As we pass from the level $a$ to the level $a_+$, 
the simple closed curve $m_2$ turns into one or two simple closed curves, 
which bound once-punctured disks in $P-L$. 
Note that $\pi_0(m_2)$ is isotopic to $\pi_0(l_a)$ in $\Sigma_0-L$ 
because $m_1$ bounds a disk in $\Sigma_a-L$. 
It follows that as we pass from the level $a$ to the level $a_+$, 
the simple closed curve $m_2$ turns into one or two simple closed curves 
that is inessential in $\Sigma_{a_+}-L$, and thus 
all of the simple closed curves of $f_+^{-1}(a_+)$ are inessential in $\Sigma_{a_+}-L$. 
This contradicts the fact 
that the point $(a_+,t_+)$ does not lie in $\mathscr{R}_a \cup \mathscr{R}_b$.  
This completes the proof of the inequality (\ref{eq:2}) in Case~A.  

\noindent 
\textbf{Case~B:} The simple closed curve $l_a$ cobounds with another essential simple closed curve $l'_a$ 
an annulus in $\Sigma'_{t_+}$ (and hence in $\Sigma_{a_-}$) 
that intersects $L$ in at most one point. 

Let $A$ be the annulus in $\Sigma'_{t_+}$ bounded by $l_a$ and $l'_a$ 
(note that such an annulus is unique because $(\mathrm{genus}(\Sigma),n) \not=(1,1)$). 
We note that $\mathcal{L}_a=\{l_a,l'_a\}$. 
There are five cases to consider as we pass from the level $a_-$ to the level $a$.  

\noindent 
\textbf{Case~B1:}  
A new simple closed curve $m$ is created away from $l_a$ and $l_a'$. 

This case is same as  Case~A1. 

\noindent 
\textbf{Case~B2:} The simple closed curve $l_a$ and another simple closed curve 
$c \not= l'_a$ of $f_+^{-1}(a_-)$ are pinched together to produce a single simple closed curve $m$ 
(Figure \ref{fig:configuration_of_curves_B2}).  

We see that $c$ bounds a once-punctured disk in $\Sigma_{a_-}-L$. 
Suppose, contrary to our claim, that $c$ bounds a disk in $\Sigma_a-L$. 
Then, it follows that $\pi_0(m)$ is isotopic to $\pi_0(l_a)$ in $\Sigma_0-L$. 
As we pass from the level $a$ to the level $a_+$, 
the simple closed curves $m$ and $l'_a$ are pinched together  
to produce an inessential simple closed curve in $\Sigma_{a_+}-L$. 
This contradicts the fact 
that the point $(a_+,t_+)$ does not lie in $\mathscr{R}_a \cup \mathscr{R}_b$. 

By the assumption, any meridional loop of $L$ does not bound a disk in $M-L$. 
Thus, it follows that $|A \cap L|=1$ and $c$ bounds a once-punctured disk in $A-L$. 
The possible configuration of $l_a$, $l'_a$, $c$ and $m$ in the annulus $A$ 
is shown in Figure \ref{fig:configuration_of_curves_B2}. 
\begin{figure}[htbp]
\begin{center}
\includegraphics[width=8.5cm,clip]{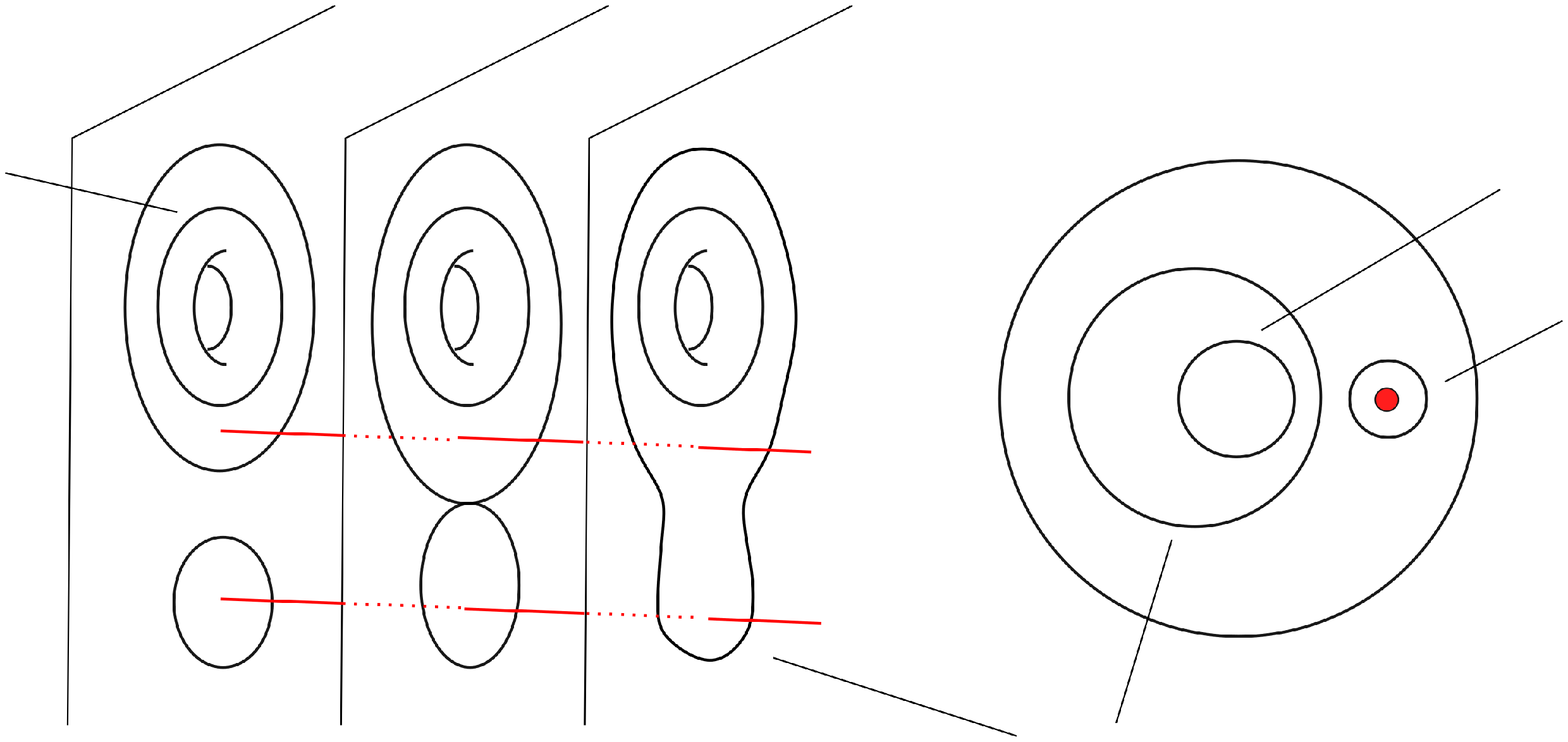}
\begin{picture}(400,0)(0,0)
\put(95,56){$l_a$}
\put(225,50){$l_a$}
\put(239,10){$m$}
\put(68,98){$l'_a$}
\put(313,97){$l'_a$}
\put(98,25){$c$}
\put(324,76){$c$}
\put(270,35){$A$}
\put(115,110){$\Sigma_{a_-}$}
\put(194,110){$\Sigma_a$}
\end{picture}
\caption{Case~B2.}
\label{fig:configuration_of_curves_B2}
\end{center}
\end{figure}

As we pass from the level $a$ to the level $a_+$,  
the simple closed curves $l_a'$ and $m$ are pinched together to produce a new single curve $m'$. 
The simple closed curve $m'$ is essential in $\Sigma_{a_+}-L$ because the point $(a_+,t_+)$ lies in the complement 
of $\mathscr{R}_a \cup \mathscr{R}_b$. 
On the other hand, $m'$ bounds a disk $D$ in $A-L$. 
By repeatedly compressing $D$ along the innermost disk in $\Sigma_{a_+}-L$, 
we obtain a disk $D'$ in $V_{a_+}^+-L$ such that 
$\partial D'=m'$. 
As $d_{\mathcal{C}(\Sigma_{L})}(\pi_0(l'_a),\pi_0(m')) \le 1$, 
it follows that $d_{\mathcal{C}(\Sigma_{L})}(\pi_0(l'_a),\mathcal{PD}(V_L^+)) \le 1$. 

\noindent 
\textbf{Case~B3:} The simple closed curve $l_a$ passes through a puncture and turns into a new simple closed curve $m$ 
(Figure \ref{fig:configuration_of_curves_B3}). 

In the annulus $A$, $m$ cobounds with $l'_a$ an annulus. 
See Figure \ref{fig:configuration_of_curves_B3}.  
\begin{figure}[htbp]
\begin{center}
\includegraphics[width=8cm,clip]{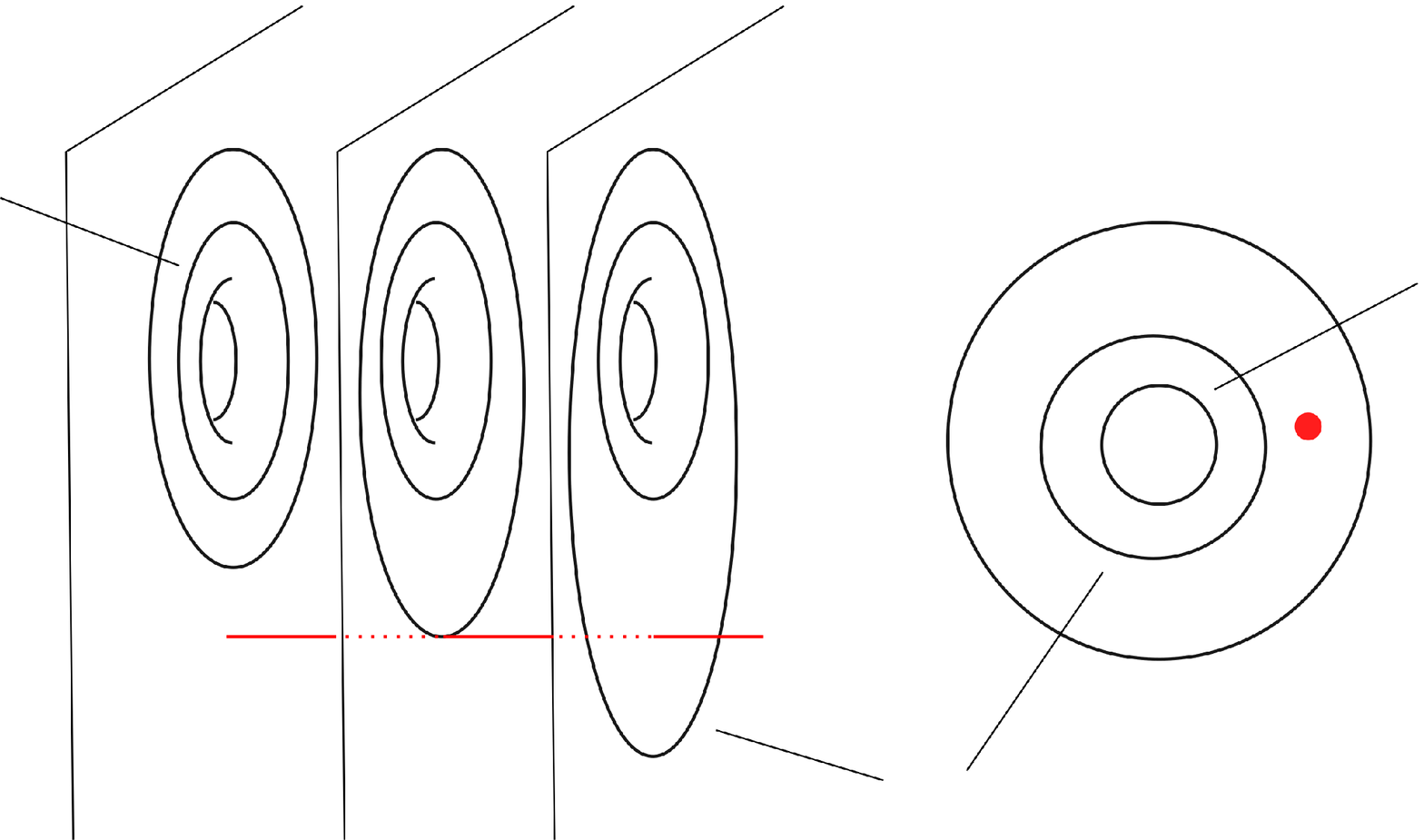}
\begin{picture}(400,0)(0,0)
\put(103,60){$l_a$}
\put(230,60){$l_a$}
\put(232,20){$m$}
\put(80,112){$l'_a$}
\put(315,102){$l'_a$}
\put(280,49){$A$}
\put(120,130){$\Sigma_{a_-}$}
\put(200,130){$\Sigma_a$}
\end{picture}
\caption{Case~B3.}
\label{fig:configuration_of_curves_B3}
\end{center}
\end{figure}
As we pass from the level $a$ to the level $a_+$,   
the simple closed curves $l_a'$ and $m$ are pinched together to produce a new single curve $m'$. 
The simple closed curve $m'$ is essential in $\Sigma_{a_+}-L$ because the point $(a_+,t_+)$ lies in the complement 
of $\mathscr{R}_a \cup \mathscr{R}_b$. 
On the other hand, $m'$ bounds a disk in $A-L$. 
By repeatedly compressing $D$ along the innermost disk in $\Sigma_{a_+}-L$, 
we obtain a disk $D'$ in $V_{a_+}^+-L$ such that 
$\partial D'=m'$. 
Therefore, we have $d_{\mathcal{C}(\Sigma_{L})}(\pi_0(l'_a),\mathcal{PD}(V_L^+)) \le 1$. 
 
\noindent 
\textbf{Case~B4:} The simple closed curve $l_a$ is pinched to produce two simple closed curves $m_1$ and $m_2$ 
(Figure \ref{fig:configuration_of_curves_B4}). 

There are two possible configurations of $l_a$, $l'_a$, $m_1$ and $m_2$ in the annulus $A$.  
See Figure \ref{fig:configuration_of_curves_B4}. 
\begin{figure}[htbp]
\begin{center}
\includegraphics[width=9cm,clip]{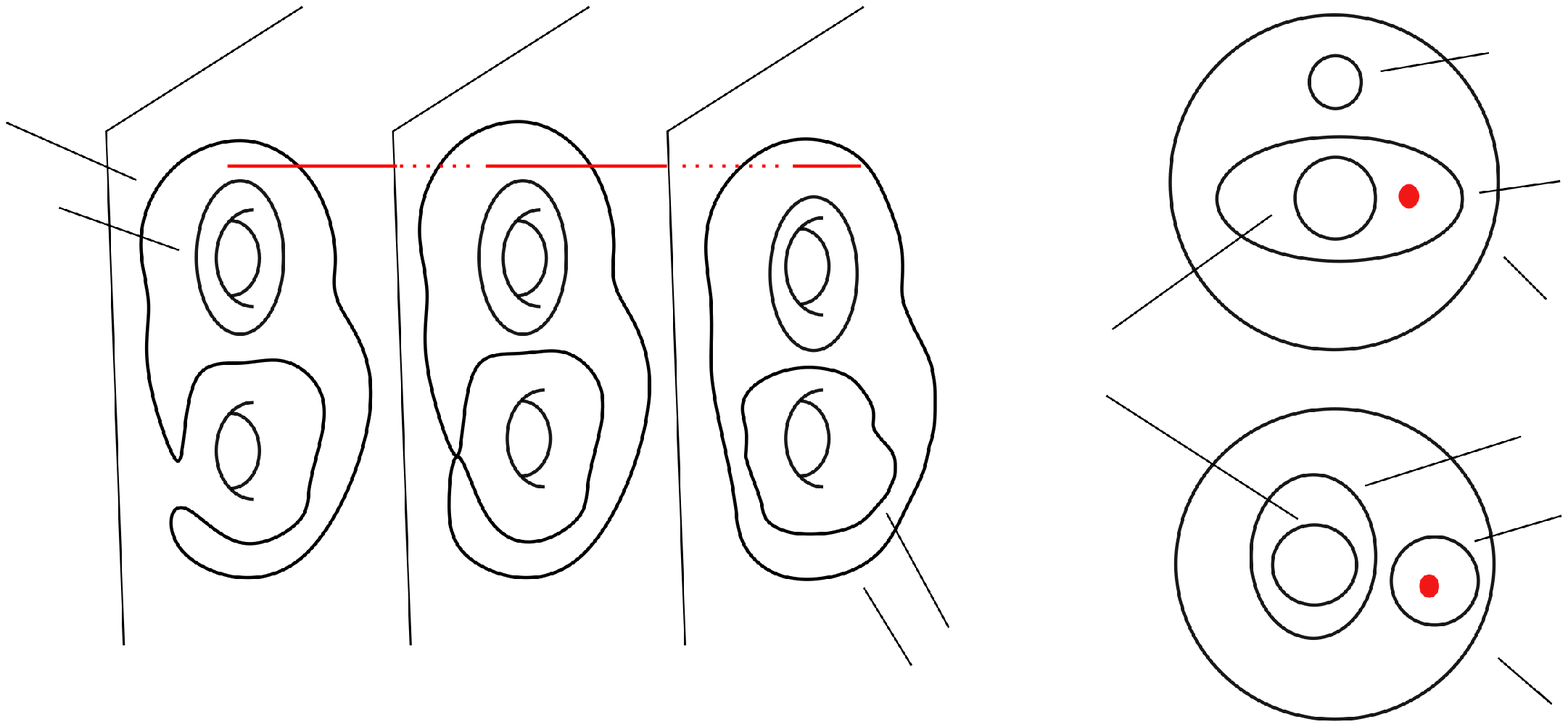}
\begin{picture}(400,0)(0,0)
\put(63,112){$l_a$}
\put(327,74){$l_a$}
\put(330,7){$l_a$}
\put(70,96){$l'_a$}
\put(245,67){$l'_a$}
\put(230,20){$m_1$}
\put(220,10){$m_2$}
\put(330,47){$m_1$}
\put(325,57){$m_2$}
\put(317,120){$m_1$}
\put(332,100){$m_2$}
\put(290,15){$A$}
\put(285,75){$A$}
\put(118,120){$\Sigma_{a_-}$}
\put(210,120){$\Sigma_a$}
\end{picture}
\caption{Case~B4.}
\label{fig:configuration_of_curves_B4}
\end{center}
\end{figure}

First, suppose that $m_1$ bounds a disk $D$ in $A-L$. 
By the assumption, any meridional loop of $L$ does not bound a disk in its complement. 
Thus, the curve $m_1$ does not bound a once-punctured disk in $\Sigma_a-L$. 
We claim that $m_1$ does not bound a disk in $\Sigma_a-L$. 
Suppose, contrary to our claim, that $m_1$ bounds a disk in $\Sigma_a-L$. 
Then, $\pi_0(m_2)$ is isotopic to $\pi_0(l_a)$ in $\Sigma_0-L$. 
As we pass from the level $a$ to the level $a_+$, 
the simple closed curves $m_2$ and $l'_a$ are pinched together to produce a single simple closed curve 
that is inessential in $\Sigma_a-L$. 
This contradicts the fact that the point $(a_+,t_+)$ does not lie in $\mathscr{R}_a \cup \mathscr{R}_b$. 
Therefore, we conclude that $m_1$ must be essential in $\Sigma_a-L$. 

By repeatedly compressing $D$ 
along the innermost disk with at most one puncture in $\Sigma_a-L$ as long as possible, 
we finally obtain a disk $D'$ in $V^+_a-L$ such that 
$\partial D' = m_1$. 
Since $d_{\mathcal{C}(\Sigma_{L})}(\pi_0(l'_a),\pi_0(m_1)) \le 1$, 
we have 
$d_{\mathcal{C}(\Sigma_{L})}(\pi_0(l'_a),\mathcal{PD}(V_L^+)) \le 1$. 

Next, suppose that $m_1$ bounds a once-punctured disk $D$ in $A-L$. 
If $m_1$ is essential in $\Sigma_a-L$, 
by repeatedly compressing $D$ 
along the innermost disk with at most one puncture in $\Sigma_a -L$ as long as possible, 
we finally obtain a disk $D'$ in the handlebody $V^+_{a}$ such that 
$\partial D' = m_1$ and $|D' \cap L| = 1 $. 
Since $d_{\mathcal{C}(\Sigma_{L})}(\pi_0(l_a),\pi_0(m_1)) \le 1$, 
it follows that $d_{\mathcal{C}(\Sigma_{L})}(\pi_0(l'_a),\mathcal{PD}(V_L^+)) \le 1$.  
Thus, in the following, we shall assume that $m_1$ is inessential in $\Sigma_a-L$. 

The simple closed curve $m_2$ cobounds an annulus with $l'_a$ in $A-L$. 
As we pass from the level $a$ to the level $a_+$, 
the simple closed curves $l'_a$ and $m_2$ are pinched together 
to produce a new single simple closed curve $m'$. 
The curve $m'$ is essential in $\Sigma_{a_+}-L$ because the point $(a_+,t_+)$ lies in the complement 
of $\mathscr{R}_a \cup \mathscr{R}_b$. 
On the other hand, $m'$ bounds a disk $D$ in $A-L$. 
By repeatedly compressing $D$ 
along the innermost disk in $\Sigma_{a_+} -L$ as long as possible, 
we finally obtain a disk $D'$ in $V_{a_+}^+-L$ such that 
$\partial D' = m'$. 
Since $d_{\mathcal{C}(\Sigma_{L})}(\pi_0(l'_a),\pi_0(m')) \le 1$, 
it follows that $d_{\mathcal{C}(\Sigma_{L})}(\pi_0(l'_a),\mathcal{PD}(V_L^+)) \le 1$. 

\noindent
\textbf{Case~B5:} 
The simple closed curves $l_a$ and $l'_a$ are pinched together   
to produce a single simple closed curve $m$ 
(Figure \ref{fig:configuration_of_curves_B5}).  

Since the point $(a,t_+)$ is in the complement in $[-1,1] \times [-1,1]$ of $\mathscr{R}_{a} \cup \mathscr{R}_{b}$, 
$m$ is essential in $\Sigma_a-L$. 
In $A-L$, $m$ bounds a disk $D$ with at most one puncture. 
See Figure \ref{fig:configuration_of_curves_B5}. 
\begin{figure}[htbp]
\begin{center}
\includegraphics[width=8cm,clip]{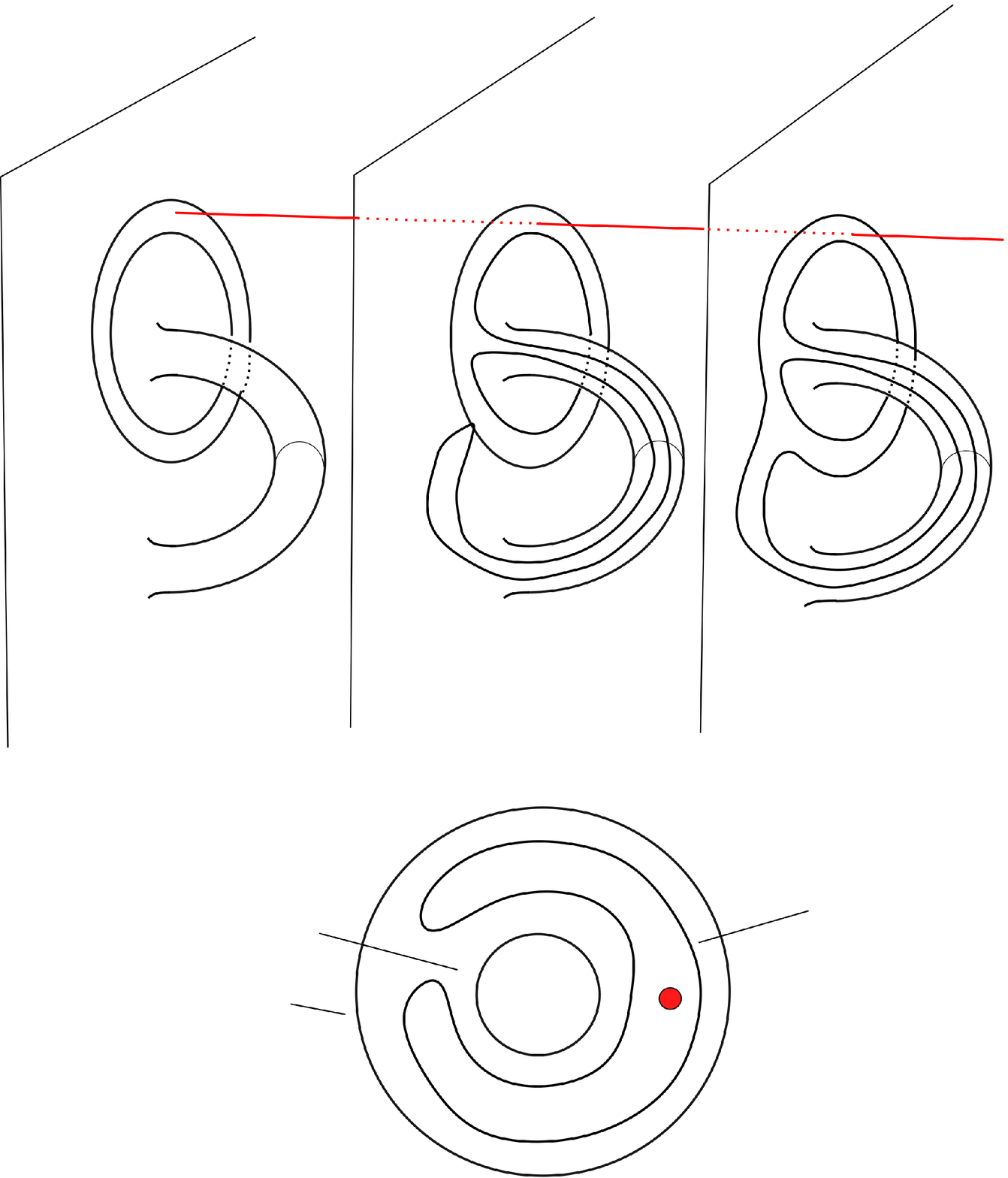}
\begin{picture}(400,0)(0,0)
\put(98,200){$l_a$}
\put(145,50){$l_a$}
\put(272,72){$m$}
\put(298,210){$m$}
\put(122,213){$l'_a$}
\put(150,65){$l'_a$}
\put(123,245){$\Sigma_{a_-}$}
\put(285,245){$\Sigma_a$}
\end{picture}
\caption{Case~B5.}
\label{fig:configuration_of_curves_B5}
\end{center}
\end{figure}
By repeatedly compressing $D$ 
along the innermost disk with at most one puncture in $\Sigma_{a} -L$ as long as possible, 
we finally obtain a disk $D'$ in the handlebody $V^+_{a}$ such that 
$\partial D' =m$ and $|D' \cap L| \leq 1 $. 
Therefore, we have 
$d_{\mathcal{C}(\Sigma_{L})}(\pi_0(l_a),\mathcal{PD}(V_L^+)) \le d_{\mathcal{C}(\Sigma_{L})}(\pi_0(l_a),\pi_0(m)) \le 1$, 
which completes the proof of the inequality (\ref{eq:2}) in Case~B. 

The symmetric argument of the proof of the inequality (\ref{eq:2}) shows the following inequality: 
\begin{eqnarray}\label{eq:3}
d_{\mathcal{C}(\Sigma_{L})}(\pi_0(\mathcal{L}_b),\mathcal{PD}(V^{-}_{L})) \le 1.
\end{eqnarray}
By the inequalities (\ref{eq:1'}), (\ref{eq:2}) and (\ref{eq:3}), for some $l_a \in \mathcal{L}_a$ and $l_b \in \mathcal{L}_b$  
we have 
\begin{align*}
d_{\mathcal{PD}}(M,L;\Sigma)  &\le  
d_{\mathcal{C}(\Sigma_{L})}(\mathcal{PD}(V^{-}_{L}),\pi_0(l_{b})) 
+ d_{\mathcal{C}(\Sigma_{L})}(\pi_0(l_{b}),\pi_0(l_{a}))) \nonumber +d_{\mathcal{C}(\Sigma_{L})}(\pi_0(l_{a}),\mathcal{PD}(V^{+}_{L})) \nonumber \\
&\le 1+1+1 \le 3.
\end{align*}
This completes the proof of Lemma~\ref{lem:split}. 
\end{proof}

We now complete the proof of Theorem~\ref{thm:injection}. 
By Lemma~\ref{lem:split}, $g_{r}$ must span $f$ for all $r \in [0,1]$. 
Lemma~\ref{lem:span} says that 
$\phi|_{\Sigma}$ is isotopic to $\mathrm{id}|_{\Sigma}$ relative to the points $\Sigma \cap L$. 
Therefore, $\phi$ represents the trivial element in $\mathcal{G}(M,L;\Sigma)$, which 
implies that the map $\eta$ is injective. 

\section{Proof of Theorem~$\ref{thm:main}$}
\label{sec:proof_of_main_thm}

We are now in position to prove Theorem~$\ref{thm:main}$.  

\vspace{0.5em}

\begin{proof}[Proof of Theorem~$\ref{thm:main}$.] 
Let $(M,L;\Sigma)$ be a bridge decomposition of $L$ with the distance at least $6$, 
where $L$ is a link in a $3$-manifold $M$. 
By Theorem \ref{thm:hyperbolic_structure}, $M_L$ admits a complete and finite volume hyperbolic structure. 
In this case it is well known that its mapping class group $\mathrm{MCG}(M_L)$ is a finite group,
hence so is $\mathrm{MCG}(M,L)$. 
The first assertion now follows from Theorem~\ref{thm:injection} and the inequality (\ref{eq:d_PD}). 
The second assertion can be shown by the same argument using 
Proposition \ref{prop:d_PD} instead of the inequality (\ref{eq:d_PD}). 
\end{proof}
 
 \begin{ack}
The authors would like to thank Kazuhiro Ichihara and Yeonhee Jang for 
their helpful comments on the remark at the end of Section \ref{sec:preliminaries_dist}. 
The authors would like to thank the anonymous referee
for carefully reading the manuscript. 
D. I. is supported by JSPS KAKENHI Grant
Number JP21J10249. 
Y. K. is supported by JSPS KAKENHI Grant
 Numbers JP20K03588, JP20K03614 and JP21H00978. 
\end{ack}


\begin{thebibliography}{99999}

\bibitem{BS05}
D. Bachman and S. Schleimer, 
Distance and bridge position. 
Pacific J. Math. {\bf 219} (2005), no. 2, 221--235. 

\bibitem{BCJTT17}
R. Blair, M. Campisi, J. Johnson, S. A. Taylor and M. Tomova, 
Exceptional and cosmetic surgeries on knots, 
Math. Ann. {\bf 367} (2017), no. 1-2, 581--622. 

\bibitem{Hak68}
W. Haken, 
Some results on surfaces in $3$-manifolds, {\it Studies in Modern Topology} pp. 39--98 Math. Assoc. Amer. 
(distributed by Prentice-Hall, Englewood Cliffs, N.J.) 1968. 

\bibitem{Hart02}
K. Hartshorn,  
Heegaard splittings of Haken manifolds have bounded distance. 
Pacific J. Math. {\bf 204} (2002), no. 1, 61--75.

\bibitem{Hem01}
J. Hempel, 
$3$-manifolds as viewed from the curve complex, Topology {\bf 40} (2001), no. 3, 631--657.

\bibitem{HIKK21}
S. Hirose, D. Iguchi, E. Kin and Y. Koda, 
Goeritz groups of bridge decompositions, 
to appear in Int. Math. Res. Not. 

\bibitem{IM17}
K. Ichihara and J. Ma, 
A random link via bridge position is hyperbolic, Topology Appl. {\bf 230} (2017), 131--138.

\bibitem{Jan14}
Y. Jang, 
Distance of bridge surfaces for links with essential meridional spheres. 
Pacific J. Math. {\bf 267} (2014), no. 1, 121--130.

\bibitem{Joh10_PAMS}
J. Johnson,  
Mapping class groups of medium distance Heegaard splittings, 
Proc. Amer. Math. Soc. {\bf 138} (2010), no. 12, 4529--4535.

\bibitem{Joh10_JTP}
J. Johnson, 
Bounding the stable genera of Heegaard splittings from below. 
J. Topol. {\bf 3} (2010), no. 3, 668--690.
 
\bibitem{Joh11}
J. Johnson, 
Automorphisms of the three-torus preserving a genus-three Heegaard splitting. 
Pacific J. Math. {\bf 253} (2011), no. 1, 75--94. 

\bibitem{KS00}
T. Kobayashi and O. Saeki, 
The Rubinstein-Scharlemann graphic of a 3-manifold as the discriminant set of a stable map. 
Pacific J. Math. {\bf 195} (2000), no. 1, 101--156.

\bibitem{KP10}
M. Korkmaz and A. Papadopoulos, 
On the arc and curve complex of a surface. 
Math. Proc. Cambridge Philos. Soc. {\bf 148} (2010), no. 3, 473--483.

\bibitem{Nam07}
H. Namazi, Big Heegaard distance implies finite mapping class group, Topology Appl.
{\bf 154} (2007), 2939--2949.
 
\bibitem{RS96}
H. Rubinstein and M. Scharlemann, 
Comparing Heegaard splittings of non-Haken $3$-manifolds,  
Topology {\bf 35} (1996), 1005--1026.

\bibitem{ST06}
M. Scharlemann and M. Tomova, 
Alternate Heegaard genus bounds distance. 
Geom. Topol. {\bf 10} (2006), 593--617.

\bibitem{Tom07}
M. Tomova,
Multiple bridge surfaces restrict knot distance. 
Algebr. Geom. Topol. {\bf 7} (2007), 957--1006.

\end{thebibliography}
\end{document}